\documentclass[12pt]{article}

\usepackage[utf8]{inputenc}
\usepackage{xcolor}
\usepackage{amsfonts}
\usepackage{bbm}
\usepackage{appendix}
\usepackage{amsmath}
\usepackage{amssymb}
\usepackage{setspace} 
\usepackage{indentfirst} 
\usepackage{hyperref}
\usepackage[colors,economics]{optsys}
\usepackage[labelfont=bf]{caption}
\usepackage{subcaption}
\usepackage{tabularx}
\usepackage{algorithm}
\usepackage{algpseudocode}
\usepackage[autolanguage]{numprint} 
\usepackage{siunitx}
\usepackage{subcaption}
\usepackage{empheq}
\usepackage{dsfont}
\usepackage{cases}

\hypersetup{
  colorlinks=true, 
  linktoc=all,     
  linkcolor=blue,
  urlcolor=blue,
  citecolor=blue
}

\usepackage{wasysym} 
\usepackage{stmaryrd} 
\usepackage{tikz}
\usetikzlibrary{fit}
\usepackage{float}
\usetikzlibrary{automata,arrows,positioning,calc}

\newcommand{\Compressor}{\textsc{c}}

\newcommand{\ELECTROLYSERFUNCTION}{\Phi^{\Electrolyser}}
\newcommand{\Electricity}{E}

\newcommand{\Electrolyser}{\textsc{e}}

\newcommand{\Grid}{\textsc{\tiny G}}
\newcommand{\HOUR}{\mathbb{H}}

\newcommand{\Hydrogen}{H}

\newcommand{\InDemand}{\rightarrow\textsc{\tiny \Demand}}

\newcommand{\PPA}{\textsc{\tiny PPA}}

\newcommand{\PV}{\textsc{\tiny PV}}

\newcommand{\hour}{h}

\newcommand{\MaximalProductionHour}{\overline{m}^\Electrolyser}

\newcommand{\load}{\ell^{\Electrolyser}}
\newcommand{\Mode}{M^{\Electrolyser}}

\newcommand{\vaload}{\boldsymbol{\ell}^{\Electrolyser}}
\newcommand{\vaMode}{\va{M}^{\Electrolyser}}

\newcommand{\TurnElectrolyser}{M^{\Electrolyser^\curvearrowleft}}
\newcommand{\vaTurnElectrolyser}{\va{M}^{\Electrolyser^{\curvearrowleft}}}

\newcommand{\IDLE}{\textsc{idle}}
\newcommand{\START}{\textsc{start}}
\newcommand{\COLD}{\textsc{cold}}
\newcommand{\MODES}{\na{\COLD,\IDLE,\START}}
\newcommand{\SHORTMODES}{{\mathbb{M}}}

\newcommand{\OP}{\textsc{\tiny O}}
\newcommand{\EL}{\textsc{\tiny E}}

\newcommand{\MaximalProduction}{\mu}

\newcommand{\MinLoad}{\underline{\ell}^{\Electrolyser}}

\newcommand{\LoadSetFunction}{\mathcal{L}}

\newcommand{\HydrogenProduced}{\Hydrogen^{\Electrolyser}}

\newcommand{\vaHydrogenProduced}{\va{\Hydrogen}^{\Electrolyser}}

\newcommand{\PPAStock}{P}
\newcommand{\ElectricityCumul}{Q}

\graphicspath{{./}}
\usepackage{a4wide}
\usepackage{fullpage}

\def\keywords#1{\par\addvspace\medskipamount{\rightskip=0pt plus1cm
\def\and{\ifhmode\unskip\nobreak\fi\ $\cdot$
}\noindent\keywordname\enspace\ignorespaces#1\par}}
\def\keywordname{{\bfseries Keywords}}%

\title{Multistage stochastic optimization \\ of a mono-site hydrogen infrastructure \\ by decomposition techniques}

\author{Raian Lefgoum\thanks{Cermics, École des Ponts ParisTech, 6 et 8 avenue Blaise Pascal, 77455 Marne la Vallée Cedex 2}
  \and
  Sezin Afsar\thanks{Universidad de Oviedo, Gijón, Principality of Asturias, Spain}
  \and Pierre Carpentier\thanks{UMA, ENSTA Paris, IP Paris, France}
  \and Jean-Philippe Chancelier$^{*}$
  \and Michel De Lara$^{*}$}

\begin{document}
\maketitle

\begin{abstract}
  The development of hydrogen infrastructures requires to reduce their costs.
  In this paper, we develop a multistage stochastic optimization model for the
  management of a hydrogen infrastructure which consists of an electrolyser, a
  compressor and a storage to serve a transportation demand. This infrastructure
  is powered by three different sources: on-site photovoltaic panels (PV),
  renewable energy through a power purchase agreement (PPA) and the power grid. We
  consider uncertainties affecting on-site photovoltaic production and hydrogen
  demand.  Renewable energy sources are emphasized in the hydrogen production
  process to ensure eligibility for a subsidy, which is awarded if the proportion
  of nonrenewable electricity usage stays under a predetermined threshold. We
  solve the multistage stochastic optimization problem using a decomposition
  method based on Lagrange duality. The numerical results indicate that the
  solution to this problem, formulated as a policy, achieves a small duality gap,
  thus proving the effectiveness of this approach.
\end{abstract}

\keywords{Hydrogen infrastructure \and Stochastic optimization \and Lagrange decomposition}


\section{Introduction}\label{Schiever_introduction}
Hydrogen, a versatile energy carrier, is predominantly produced using fossil
fuels, for example through steam methane reforming (SMR) of natural gas, a
process that releases significant carbon dioxide (CO2) emissions. As concerns
over climate change intensify, there is a growing imperative to shift towards
cleaner methods of hydrogen production. Water electrolysis, which generates
hydrogen by splitting water molecules using electricity, offers a promising
solution. However, the environmental benefits of hydrogen are contingent upon
the use of renewable electricity sources.

In addition to environmental considerations, the economic viability of the
hydrogen produced by water electrolysis hinges significantly on electricity
costs. The cost of electricity represents a significant proportion of the total
operational expenses~\cite{ELECCOST}, which presents a significant challenge to
the widespread adoption of this technology. Thus, optimizing the cost of
electricity becomes essential to ensure the competitiveness of hydrogen.

Hydrogen production has several characteristic features. Firstly, the
electrolyser, which is the equipment used to produce hydrogen, has a nonlinear
electrical consumption given the quantity of hydrogen produced. Typically, the
pressure of hydrogen produced by the electrolyser is quite low and since it is
the lightest element, a small quantity of it fills up a large space. Therefore,
it is customarily compressed and stored at a higher pressure. Furthermore,
hydrogen production, which requires electricity, is achieved through an energy
mix primarily composed of renewable energy sources. This has led to the recent
emergence of various contracts, known as Power Purchase Agreements (PPA), to
provide consumers with greater flexibility in producing green
hydrogen.
However, the optimization of hydrogen production faces significant
barriers. Renewable energy sources, such as solar or wind power, while abundant,
are characterized by inherent variability and intermittency. Moreover, the
uncertainty surrounding hydrogen demand adds complexity to production
optimization. Aligning this demand variability with renewable energy
intermittency poses a significant challenge in achieving efficient and
sustainable hydrogen production.

In the literature related to hydrogen management, a large number of studies
incorporates uncertainties in optimization problems by relying on stochastic
programming. More precisely,
in~\cite{TWOSTAGES1,TWOSTAGES2,TWOSTAGES3,TWOSTAGES4,TWOSTAGESELEC,TWOSTAGESARC}, two-stage stochastic
models are proposed where the first stage decision is a design decision
concerning the supply chain (equipment sizes, capacities, contracts, etc.) and the second stage is a management decision.

All the previous described works are modeled as two-stage
stochastic programming problems, and most of them are numerically solved through the use of mixed integer linear
programming solvers (MILP). However, it is worth noting that solving MILP models
becomes computationally intractable as the number of scenarios increases. As a
result, the operations aspect of the problems are simplified, leading to a
representation that may not fully capture the complexity of real-world problems. 

Few studies, like \cite{SCENARIOREDUC,SDDiP} rely on multistage stochastic optimization models taking into
account nonanticipativity constraints. In~\cite{SCENARIOREDUC}, a fast backward
scenario reduction algorithm \cite{FastReduc} is used to derive twenty representative
scenarios and then the optimization problem is solved as a MILP. In~\cite{SDDiP}, the hydrogen is
produced using renewable energies and is later converted using a fuel cell to
satisfy electricity demand. The optimization problem, formulated as a
multistage stochastic optimization problem, is solved using Stochastic Dual Dynamic Integer Programming (SDDiP)~\cite{SDDiPAlgorithm}.

To the best of our knowledge, \cite{SIMILAR} appears to be the most closely
aligned with our work. In this study, the authors employ dynamic programming to
address a multistage stochastic problem. Their work revolves around hydrogen
production utilizing both wind power and grid electricity, with subsequent
conversion back into electricity. This process serves the dual purpose of
meeting Power Purchase Agreement obligations and potentially generating revenue
through grid sales.

In this paper, we formulate and solve a multistage stochastic optimization
problem to manage a hydrogen infrastructure. Our contributions are
the following.
\begin{enumerate}
\item Compared with two-stage stochastic programming models, the proposed model
  adequately considers the sequential decisions (every hour) with the gradual
  revealing of the uncertainty over time.
\item The nonlinear electricity consumption of the electrolyser and its
  functioning modes are taken into consideration.
\item An electricity mix of on-site photovoltaic, renewable electricity through
  PPA and power grid is used to supply the
  infrastructure. Renewable energy sources are prioritized in the hydrogen
  production process to ensure eligibility for a subsidy.
\item Leveraging Lagrange duality, we decompose the original problem into two separate problems. The first one, that we call \emph{the operational problem}, involves the management of the hydrogen equipment and the demand satisfaction. The second one, that we call \emph{the electricity allocation problem}, is related to the allocation of the electricity sources. 
\end{enumerate}

The paper is organized as follows. In Sect.~\ref{Schiever:characteristics}, we
describe the studied system. In
Sect.~\ref{Schiever_problem_formulation_and_resolution}, we give the problem
formulation and propose a method based on Lagrange duality to solve the
problem. In Sect.~\ref{Schiever:Numerical_results}, we give the numerical
results. Finally, we provide the conclusion of this work in
Sect.~\ref{Schiever_conclusion}. We relegate proofs and technical points in
Appendix.

\section{Hydrogen infrastructure management problem}\label{Schiever:characteristics}
The following work is motivated by a real-life optimization problem proposed by a company. More precisely, we want to take into account in the modeling process that first, the company owns a fleet of diesel trucks and aims to decarbonize it, and second, they have several buildings with large roofs suitable for photovoltaic installations and has been awarded an investment subsidy to help develop green hydrogen infrastructure.
In~\S\ref{Schiever:characteristics1}, we describe the characteristics of the hydrogen infrastructure under study and, in~\S\ref{Schiever:mathematical_modeling}, we give its mathematical description.

\subsection{Hydrogen infrastructure case study}\label{Schiever:characteristics1}
We consider the hydrogen infrastructure described in
Figure~\ref{fig:diagram_Schiever} that is specific to the case study and 
which is composed of an electrolyser, a compressor that compresses the hydrogen
to an adequate pressure and a storage that stores the compressed hydrogen in the
same site. This hydrogen infrastructure is powered by solar panels, PPA and the
grid to satisfy an uncertain hydrogen demand.
\begin{figure}[htpp]
  \centering
  \mbox{\includegraphics[width=0.60\textwidth]{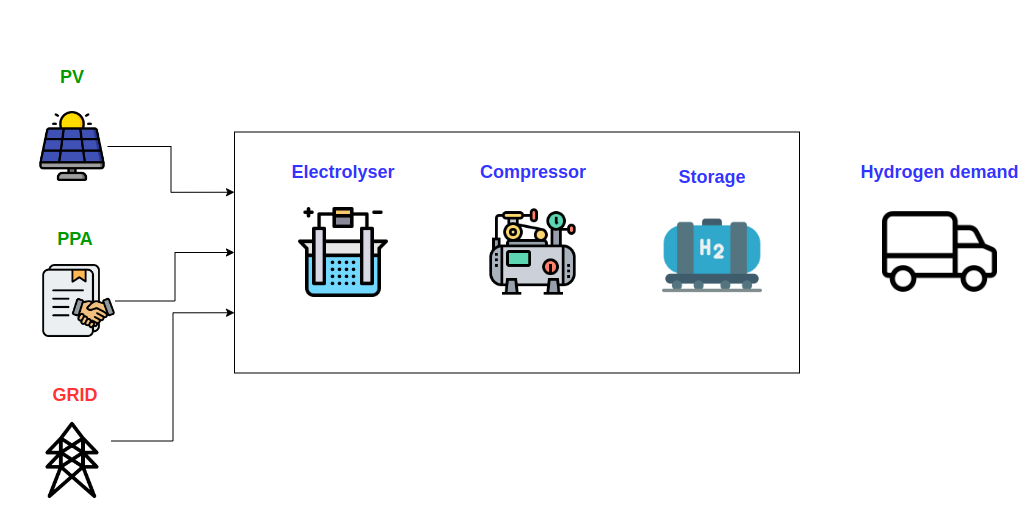}}
  \caption{Diagram of the hydrogen infrastructure (using free icons from \url{Flaticon.com})}
  \label{fig:diagram_Schiever}
\end{figure}

The optimization is done during one week (168
hours), by making decision every hour. Thus, the time horizon is $\horizon=168$,
the set of hours without the time horizon is $ \HOUR = \{0,1,..,\horizon-1\}$, the set of hours with the time horizon is $\overline{\HOUR} = \HOUR \cup \{\horizon\}$ and a generic hour is
denoted by $\hour \in \HOUR$.

In what follows, we describe the main components of the studied infrastructure: the electricity sources, the electrolyser, the volumetric compressor, the gaseous storage and the hydrogen demand.

\subsubsection{Electricity sources}
The hydrogen infrastructure is powered by the three following electricity sources.
\begin{itemize}
\item Photovoltaic (PV): energy produced from solar radiation through
  photovoltaic solar panels or power plants. The PV source is costless and its
  production is uncertain.
\item Power Purchase Agreement (PPA): is a long-term contract between an energy producer and a purchaser. This agreement outlines the terms under which the producer will sell electricity to the purchaser. PPAs are widely used in the renewable energy sector to secure financing for projects by providing a stable revenue stream. They mitigate risks for both producers and consumers, ensuring a reliable market for the energy produced. In this case study, we focus on PPA \emph{pay as consumed}, where the purchaser pays a price $c^s$ per kWh of electricity consumed, rather than a pre-agreed volume, and where, a maximal cumulated quantity of consumed electricity $\overline{\Electricity^{\PPA}}$ is fixed over a given timespan.
\item Grid: electricity available through purchase from the
  electricity network. In this study, the price of grid electricity is deterministic. Unlike the other two sources, the grid supplies electricity generated from nonrenewable sources.
\end{itemize}

\subsubsection{Electrolyser}
An electrolyser is a system made up of “stacks” (of cells) and of a BOP
(Balance of plant: rectifier/purifiers for water and gases, etc.).
Cell stacks convert chemical energy into electricity, and vice versa, by means of electrochemical reactions involving an anode and a cathode. The electrochemical
reaction that produces hydrogen (H$_2$) and oxygen (O$_2$) from water (H$_2$O)
occurs in each of the cells of the stacks.
The set of possible modes of an electrolyser is denoted by
\begin{equation}
  {\SHORTMODES} = \MODES \eqfinv
  \label{def:modes}
\end{equation}
where the three possible modes are defined as follows.
\begin{enumerate}  
\item \COLD:  in this mode, the
  electrolyser is off and does not consume electricity.
\item \IDLE:  in this mode, the
  electrolyser is on and consumes a constant amount of
  electricity, but does not produce hydrogen. 
  The advantage of being in idle mode is that the electrolyser can quickly switch to the \START~mode.
\item \START: in this mode, the electrolyser is on and is able to produce hydrogen.
\end{enumerate}

All the transitions between two modes are admissible. However, a transition
from mode $M \in \SHORTMODES$ to mode $M' \in \SHORTMODES$ requires a certain amount of time which is
assumed to be less than the timestep we consider in the modeling (1 hour). For
describing this amount of time, we introduce a function
$\mu: \SHORTMODES\times \SHORTMODES \to [0,1] $ which is such that
$1-\mu(M,M')$ quantifies the proportion of the current timestep
occupied by the transition from $M$ to $M'$. An example of a function $\mu$ for a specific electrolyser is given in Table~\ref{table:Schiever_transition_time_electrolyser} in~\S\ref{shriver-data}.

The quantity $\MaximalProductionHour$~(kg) (used in Equation~\eqref{eq:hydrogen_production}) is the maximal quantity of hydrogen that the electrolyser can produce during one timestep (hour). The function $\ELECTROLYSERFUNCTION$ (used in Equation~\eqref{eq:E_max} and~\eqref{eq:electrolyser_electricity}) gives the \emph{unitary electricity consumption} of the electrolyser, which is the electricity consumption per kilogram of hydrogen produced as a function of the load $\load$ (quantity of hydrogen produced as a percentage of the maximal hydrogen production). An example of a function $\ELECTROLYSERFUNCTION$ for a specific electrolyser is given in Figure~\ref{fig:schiever_load} in~\S\ref{shriver-data}.

\subsubsection{Volumetric compressor}
A hydrogen volumetric compressor is a device used to increase the pressure of hydrogen gas by reducing its volume through mechanical means, such as piston or diaphragm compression. This process enables the efficient storage and transportation of hydrogen in high-pressure tanks. The quantity $e^\Compressor$ (used in Equation~\eqref{eq:compressor_electricity}) is the unitary electricity consumption of the compressor per kg of hydrogen produced.

\subsubsection{Gaseous storage}
Gaseous storage of hydrogen refers to the containment of hydrogen gas under high
pressure in specialized tanks or cylinders for use in various applications such
as fuel cells, industrial processes, and hydrogen-powered vehicles. Each storage
is characterized by its volume, its temperature, and the maximal/minimal possible quantity of hydrogen that can be stored, denoted respectively by $\overline{\Stock}$ and $\underline{\Stock}$ (used Equation~\eqref{eq:stock_bounded}).

\subsubsection{Hydrogen demand $\Demand_\hour$}
The hydrogen produced is used to satisfy an uncertain hydrogen demand $\Demand_\hour$ (used in Equation~\eqref{eq:Shriver_deterministic_InstantaneousCost1} and~\eqref{eq:stock_dynamics}) in kg at each timestep $\hour$ (hour).

\subsection{Hydrogen infrastructure management modeling}\label{Schiever:mathematical_modeling}

As shown in Figure~\ref{fig:diagram_Schiever}, the operations of this hydrogen infrastructure consist of determining the electricity mix between the uncertain PV, PPA and grid, the mode of the electrolyser and the quantity of hydrogen to produce to satisfy the uncertain hydrogen demand at every hour.

In what follows, we mathematically formulate the hydrogen infrastructure management problem. For this purpose, we introduce decision, state and uncertainty variables, as well as cost functions and constraints.
\subsubsection{Decision variables}
\label{Schiever:decision_variables}

For every timestep (hour) $\hour \in \HOUR$, the decision variables of the problem are described in
Table~\ref{table:Schiever_decisions}.

\begin{table}[ht]
  \begin{tabular}{|p{1.5cm}|p{7.0cm}|c|} \hline
    \centering{Decision}  & Description & Domain\\ \hline\hline
    $\Electricity^{\PPA}_{\hour}$ & Electricity from PPA (kWh)
                                        &$\mathbb{R_+}$\\ \hline 
    $\Electricity^{\Grid}_{\hour+1}$ & Electricity from the grid (kWh)
                                        &$\mathbb{R}$\\ \hline
    $\load_\hour$& Load at which the electrolyser is functioning&$[\MinLoad,1]$\\ \hline
    $\TurnElectrolyser_\hour$& Turn the electrolyser to \COLD {,} \IDLE\ or \START\ mode
                                        & $\SHORTMODES$ (see \eqref{def:modes}) \\ \hline
    $\Hydrogen^{\InDemand}_{\hour}$ & Quantity of hydrogen extracted from the storage (kg)
                                        & $\mathbb{R_+}$ \\ \hline
  \end{tabular}
  \caption{Decision variables}
  \label{table:Schiever_decisions}
\end{table}
For every timestep (hour) $\hour \in \HOUR$, we gather all decision variables in the control vector 
\begin{equation}\label{eq:Schiever_decision}
  \Control_{\hour} = \Bp{ 
    \Electricity^{\PPA}_\hour,
    \Electricity^{\Grid}_{\hour+1},
    \load_\hour, 
    \TurnElectrolyser_\hour,
    \Hydrogen^{\InDemand}_\hour
  } \eqfinp
\end{equation}

\subsubsection{Physical state variables} 
\label{Schiever:state_variables}
For every timestep (hour) $\hour \in \overline\HOUR$, we define the physical state variables in
Table~\ref{table:Schiever_state_variables}.
\renewcommand{\arraystretch}{1.2}
\begin{table}[ht]
  \begin{tabular}{|p{1.5cm}|p{7.0cm}|c|} \hline
    \centering{State} & Description & Domain\\ \hline
    \centering$\Stock_\hour$
                      & Quantity of hydrogen in the storage (kg) &$[\underline{\Stock},\overline{\Stock}]$ \\ \hline
    \centering$\Mode_\hour$
                      & Mode of the electrolyser & $\SHORTMODES$ (see~\eqref{def:modes}) \\ \hline 
  \end{tabular}
  \caption{State variables}
  \label{table:Schiever_state_variables}
\end{table}

Economic state variables, that we name \emph{cumulative electricity} $\ElectricityCumul$ and \emph{PPA-stock} $\PPAStock$, will be defined later.

\subsubsection{Uncertain variables}\label{Schiever_uncertain_variables}
For every timestep (hour) $\hour \in \HOUR$, we define the uncertain variables in
Table~\ref{table:Schiever_uncertain_variables}.
\begin{table}[ht]
  \begin{tabular}{|p{2cm}|p{6.5cm}|c|}\hline 
    Uncertainty & Description & Domain\\ \hline\hline
    $\Electricity^{\PV}_{\hour+1}$ & Renewable (PV) electricity (kWh) during $[\hour,\hour{+}1[$&$[0,\overline{\Electricity^{\PV}_{\hour+1}}]$
    \\  \hline
    $\Demand_{\hour+1}$ & Demand of hydrogen (kg) during $[\hour,\hour+1[$& $[0,\overline{\Demand_{\hour{+}1}}]$
    \\ \hline
  \end{tabular}
  \caption{Uncertain variables}
  \label{table:Schiever_uncertain_variables}
\end{table}
\renewcommand{\arraystretch}{1}

\subsubsection{Cost functions}
\label{Schiever:cost_function}
\subsubsubsection{Hourly cost}
For every timestep (hour) $\hour \in \HOUR$, the instantaneous cost is defined as
\begin{equation}
  \begin{split}
    &\InstantaneousCost_\hour\bp{\Electricity^{\PPA}_\hour,
      \Electricity^{\Grid}_{\hour+1},
      \Hydrogen^{\InDemand}_\hour,
      \Demand_{\hour+1}}  \\
    & \hspace{2.0cm} =\underbrace{c^{\PPA}  \Electricity^{\PPA}_\hour}_{\text{PPA cost}}
    + \underbrace{c_{\hour}^{G} (\Electricity^{\Grid}_{\hour+1})_+}_{\text{Grid cost}}
    + \underbrace{ c^d  (\Demand_{\hour+1} - \Hydrogen^{\InDemand}_\hour )_{+}}_{\text{Backup  cost}}  \eqfinv
  \end{split}
  \label{eq:Shriver_deterministic_InstantaneousCost1}
\end{equation} 
where $x_+ = \max(x,0).$

The electricity cost is split into the PPA cost and the grid cost. The backup
cost is linear with respect to the unsatisfied demand, that is, when the quantity
$\Demand_{\hour+1} - \Hydrogen^{\InDemand}_\hour$ is nonnegative and is equal to
zero when the demand is satisfied, that is, when $\Demand_{\hour+1} -
\Hydrogen^{\InDemand}_\hour \le 0$.

\subsubsubsection{Subsidy cost}
We also introduce a \emph{subsidy cost}
\begin{equation} \label{eq:Schiever_final_cost}
  \widetilde{\FinalCost}
  \Bp{ \np{\Electricity^{\PPA}_\hour,\Electricity_{\hour+1}^{\Grid},\Electricity_{\hour+1}^{\PV}}_{\hour \in \HOUR}}  = -c^s\findi{[0,p]}
  \Bgp{\frac{\overbrace{\sum_{\hour \in \HOUR} (\Electricity_{\hour+1}^{\Grid})_+}^{\text{Grid electricity}}}
    {\underbrace{\sum_{\hour \in \HOUR}
        \overline{\Electricity}\wedge \np{\Electricity^{\PPA}_\hour+\Electricity_{\hour+1}^{\PV}} + (\Electricity_{\hour+1}^{\Grid})_+}_{\text{Total electricity}}}}
  \eqfinv
\end{equation}
where $(\cdot \wedge \cdot)=\min(\cdot,\cdot)$ and $c^s$ is the subsidy and $\overline{\Electricity}$ is the maximal
electricity consumption of the hydrogen infrastructure
\begin{equation}\label{eq:E_max}
   \overline{\Electricity} =
  {\ELECTROLYSERFUNCTION(1)\MaximalProductionHour + 
    e^\Compressor \MaximalProductionHour} \eqfinp
\end{equation}

The subsidy cost emphasizes the use of renewable energies. Indeed, it is equal
to $-c^s$ if the cumulated electricity consumption from the grid is less than
$100p\%$ of the cumulated total electricity consumed (that is, from the grid, PPA
and PV). This reflects the fact that a subsidy $c^s$ is granted when renewable
sources contribute to more than $100(1-p)\%$ of the cumulated total electricity
consumption.

In the presence of uncertainty, and considering
Equation~\eqref{eq:schiever-elec-division} which will be elaborated upon
subsequently, the electrical input from the grid
$\Electricity^{\Grid}_{\hour+1}$ may take negative values when selling excess
electricity through the network. However, its impact on the subsidy cost~$\widetilde{\FinalCost}$ is only accounted when electricity is purchased from the
network. Similarly, the sum of PPA electricity $\Electricity^{\PPA}_\hour$ and
PV electricity $\Electricity_{\hour+1}^{\PV}$ supplied to the electrolyser
and the compressor might exceed the upper limit of electricity
($\overline{\Electricity}$) that can be accepted by the
infrastructure. Therefore, the maximum contribution of PPA and PV electricity to the
subsidy cost is $\overline{\Electricity}$, as it is delineated in Equation~\eqref{eq:Schiever_final_cost} (by
employing max and min functions). We rewrite
Equation~\eqref{eq:Schiever_final_cost} as a function of a sum over $\hour \in
\HOUR$ to ease the use of Dynamic Programming in \S\ref{maximizing_phi_lin}
\begin{subequations}
\begin{align}
  \label{eq:final_cost_tilde}
  \widetilde{\FinalCost}
  &
  \Bp{\np{\Electricity^{\PPA}_\hour,\Electricity_{\hour+1}^{\Grid},\Electricity_{\hour+1}^{\PV}}_{\hour \in \HOUR}}
  \nonumber \\
  &= K 
  \bgp{\sum_{\hour \in \HOUR} \Bp{(1{-}p)(\Electricity_{\hour+1}^{\Grid})_+
      - p \bp{\overline{\Electricity}\wedge\na{\Electricity^{\PPA}_\hour+\Electricity_{\hour+1}^{\PV}}}}}
  \eqfinv
\end{align}
where the function $\FinalCost$ is defined by
\begin{equation}\label{eq:heavyside}
  \FinalCost = -c^s\findi{\RR_-}\eqfinv 
\end{equation}
\end{subequations}
where $\RR_- = ]-\infty,0]$ and $\findi{\RR_-}(x)$ is equal to 1 if $x \in \findi{\RR_-}$ and 0 otherwise.

We reformulate the subsidy cost in Equation~\eqref{eq:final_cost_tilde} as the
final cost $\FinalCost(\ElectricityCumul_\horizon)$, where state~$\ElectricityCumul$, that we name cumulative electricity, has the following dynamics
\begin{equation}\label{eq:cumul_dynamics}
  \begin{cases} 
    \ElectricityCumul_0 = 0 \eqfinv \\ 
    \ElectricityCumul_{\hour+1} = \ElectricityCumul_\hour +
    (1-p)(\Electricity_{\hour+1}^{\Grid})_+-p\min(\overline{\Electricity},
    \Electricity^{\PPA}_\hour+\Electricity_{\hour+1}^{\PV}) \eqsepv \forall \hour \in \HOUR
    \eqfinp
  \end{cases}
\end{equation}

\subsubsubsection{Summary table}
The parameters used in the instantaneous and subsidy cost functions are described in the following Table~\ref{table:cost-variables}.
\begin{table}[ht]
  \begin{tabularx}{\linewidth}{|c|X|c|}
    \hline \rule{0pt}{2ex}
    Variable & Description & Value \\ \hline \hline 
    $c^{\PPA}$ &Unitary cost of energy provided by PPA (\euro/kWh)&\numprint{0.075}\euro/kWh \rule{0pt}{2ex} \\ \hline \hline 
    $c_{\hour}^{G}$ &Unitary cost of energy provided by buying from Grid (\euro/kWh) at timestep \hour& Figure~\ref{fig:Grid_cost} \rule{0pt}{2ex} \\ \hline \hline 
    $c^d$ &Unitary cost of not satisfying the hydrogen demand (\euro/kg)&\numprint{5000}\euro/kg \rule{0pt}{2ex} \\ \hline \hline 
    $c^s$ &Subsidy (\euro)&\numprint{5e6}\euro \rule{0pt}{2ex} \\ \hline \hline 
    $p$ & Subsidy threshold &\numprint{0.2} \\ \hline \hline 
    $\overline{\Electricity}$ &Maximal electricity consumption of the hydrogen infrastructure (kWh) &\numprint[kWh]{1403} \\ \hline 
  \end{tabularx}
  \caption{Parameters of the instantaneous and subsidy cost functions}
  \label{table:cost-variables}
\end{table}

\subsubsection{Constraints}
In this section, we describe the electrolyser, production and electricity constraints, and we give the electrolyser mode and stock dynamics.
\label{Schiever:constraints}

\subsubsubsection{Electrolyser constraints}

For every timestep (hour) $\hour \in \HOUR$, the load decision
$\load_\hour$ and the decision $\TurnElectrolyser_\hour$ to change the electrolyser mode
 are linked one to the other. Indeed, first, if we
turn the electrolyser to \IDLE\ mode (resp. \COLD\ mode) by using
$\TurnElectrolyser_\hour=\IDLE$ (resp. $\TurnElectrolyser_\hour=\COLD$), then
the electrolyser cannot produce hydrogen and therefore the load must be equal to zero
($\load_\hour=0)$.  Second, if we turn the electrolyser to \START\ mode,
$\TurnElectrolyser_\hour=\START$, then the load is to be set in the interval~$[\MinLoad,1]$.
The coupling constraint, between the load decision and the decision to change the electrolyser mode,
is mathematically formalized as follows
\begin{align}
  \load_\hour
  &\in \LoadSetFunction(\TurnElectrolyser_\hour) \eqfinv
    \label{eq:elementary_mode_dynamics1}
    \text{ with }
    \LoadSetFunction: \SHORTMODES \ni \TurnElectrolyser_\hour
    \mapsto \begin{cases}
      \na{0} & \mbox{if } \TurnElectrolyser_\hour \in \mbox{\{\COLD,\IDLE\}} \eqfinv \\
      [\MinLoad,1] & \mbox{if } \TurnElectrolyser_\hour = \mbox{\START}
      \eqfinp
    \end{cases}
\end{align}

\subsubsubsection{Production constraints}
The hydrogen production~(in~kg) during a time interval $[h,h+1[$, is given by
\begin{equation}\label{eq:hydrogen_production}
  \HydrogenProduced_\hour =
  \load_\hour\MaximalProduction(\Mode_\hour,\TurnElectrolyser_\hour)\MaximalProductionHour \eqfinv
\end{equation}
where $\load_\hour$ is the load
decision, the function $\MaximalProduction(\Mode_\hour,\TurnElectrolyser_\hour): \SHORTMODES\times\SHORTMODES \to [0,1]$ gives the proportion -- expressed in $\%$ --
of the current time interval $[h,h+1[$ which is used for hydrogen production when the electrolyser evolves from mode
$\Mode_\hour$ to mode $\TurnElectrolyser_\hour$, and
$\MaximalProductionHour$ is the maximal quantity of hydrogen that the
electrolyser can produce during one hour.

The total electricity furnished by the three energy sources is used by both the
electrolyser and the compressor, that is
\begin{subequations}
  \begin{equation}
  \underbrace{\Electricity^{\PPA}_\hour + \Electricity^{\Grid}_{\hour+1}+ \Electricity^{\PV}_{\hour+1}}_{\text{total electricity}}
  = \underbrace{\Electricity^{\Electrolyser}_\hour}_{\text{\parbox{3cm}{electricity used by \\ the electrolyser}}}
    + \underbrace{\Electricity^{\Compressor}_\hour}_{\text{\parbox{3cm}{electricity used by \\the compressor}}} \eqfinv
    \label{eq:schiever-elec-division}
\end{equation}
where the electricity used by the electrolyser is
\begin{equation}\label{eq:electrolyser_electricity}
  \Electricity^{\Electrolyser}_\hour
  = \underbrace{\ELECTROLYSERFUNCTION(\load_\hour)\HydrogenProduced_\hour}_{\text{\parbox{4cm}{electricity used by the \\ electrolyser on \START\ mode}}}
  + \; \underbrace{e^{\IDLE} \findi{\IDLE}(\TurnElectrolyser_\hour)\mu(\Mode_\hour,\TurnElectrolyser_\hour)}_{\text{\parbox{4cm}{electricity used by the \\ electrolyser on \IDLE\ mode}}}\eqfinv
\end{equation}
and the electricity used by the compressor is
\begin{equation}\label{eq:compressor_electricity}
  \Electricity^{\Compressor}_\hour= e^{\Compressor} \HydrogenProduced_\hour \eqfinp
\end{equation}
\end{subequations}

Equation~\eqref{eq:schiever-elec-division} implies that if the electricity
generated by photovoltaic solar panels $\Electricity^{\PV}_{\hour+1}$ is high,
the surplus is sold to the grid, resulting in $\Electricity^{\Grid}_{\hour+1}$
taking negative values.

\subsubsubsection{Electricity constraints}
We now describe the electricity constraints. The first and third one are induced
constraints derived from Equation~\eqref{eq:schiever-elec-division}. The second one is a
constraint induced by the PPA contract.
\begin{itemize}

 \item 
As the right hand side of Equation~\eqref{eq:schiever-elec-division}, that is, the electricity consumption of the hydrogen infrastructure, is upper bounded by $\overline{\Electricity}$, we obtain that the left hand side of Equation~\eqref{eq:schiever-elec-division} is also upper bounded, that is
\begin{equation}\label{eq:schiever_total_elec_bounded}
  \Electricity^{\PPA}_\hour + \Electricity^{\Grid}_{\hour+1}+ \Electricity^{\PV}_{\hour+1}
  \leq \overline{\Electricity} 
  \eqfinp 
\end{equation}

  \item
The cumulated energy from PPA is upper bounded (as imposed by the contract)
\begin{equation}\label{eq:ppa}
  \sum_{\hour \in \HOUR} \Electricity^{\PPA}_\hour \leq \overline{\Electricity^{\PPA}} \eqfinv
\end{equation}
where $\overline{\Electricity^{\PPA}}$ is the maximal available quantity of PPA
electricity during the time horizon.

\item 
As $\Electricity^{\PPA}_\hour$ is upper bounded by $\overline{\Electricity^{\PPA}}$ (see Equation~\eqref{eq:ppa}), $\Electricity^{\PV}_{\hour+1}$ is upper bounded by $\overline{\Electricity^{\PV}_{\hour+1}}$ (see Table~\ref{table:Schiever_uncertain_variables}) and as the right hand side of Equation~\eqref{eq:schiever-elec-division} is nonnegative for all $\hour \in \HOUR$, we obtain that the grid electricity $\Electricity^{\Grid}_{\hour+1}$ is lower bounded for all $\hour \in \HOUR$
\begin{equation}\label{eq:schiever_grid_bounded}
 \underline{\Electricity^{\Grid}} \leq \Electricity^{\Grid}_{\hour+1} \text{ with }{\underline{\Electricity^{\Grid}} = -\overline{\Electricity^{\PPA}}-\overline{\Electricity^{\PV}}} \eqfinp
\end{equation}

\end{itemize}
To ease the use of Dynamic Programming, we reformulate Constraint~\eqref{eq:ppa}
as $\PPAStock_\horizon \geq 0$, where state $\PPAStock$, that we name PPA-stock, has the
following dynamics
\begin{equation}\label{eq:PPA_dynamcis}
    \PPAStock_0 = \overline{\Electricity^{\PPA}} \text{ and } \PPAStock_{\hour+1} = \PPAStock_\hour - \Electricity^{\PPA}_\hour \eqsepv \forall \hour \in \HOUR \eqfinp
\end{equation}

Assuming constraint~\eqref{eq:PPA_dynamcis}, and noting that
$\Electricity^{\PPA}_\hour$ is nonnegative for all $\hour \in \HOUR$, it is
immediate to see that constraint $\PPAStock_\horizon \geq 0$ is equivalent to
$\Electricity^{\PPA}_{\hour}\leq \PPAStock_\hour$ for all $\hour \in \HOUR$.

\subsubsubsection{Electrolyser mode dynamics} 
The dynamics of the mode of the electrolyser is given by 
\begin{equation}\label{eq:mode_dynamics}
    \Mode_{\hour+1} = 
     \TurnElectrolyser_\hour
\eqfinp
\end{equation}

\subsubsubsection{Stock dynamics and constraints}
The dynamics of the stock is given by
\begin{equation}\label{eq:stock_dynamics}
  \Stock_{\hour +1} = \Stock_{\hour} + \HydrogenProduced_\hour - \min(\Demand_{\hour+1},\Hydrogen^{\InDemand}_\hour) \eqfinv
\end{equation}
where $\min(\Demand_{\hour+1},\Hydrogen^{\InDemand}_\hour)$ reflects that, if the
quantity ($\Hydrogen^{\InDemand}_\hour$) extracted from the stock is greater
than the demand ($\Demand_{\hour+1}$), we re-inject the unused quantity of
hydrogen in the stock. Moreover, the stock is upper and lower bounded
\begin{equation}\label{eq:stock_bounded}
 \underline{\Stock} \leq \Stock_\hour \leq \overline{\Stock} \eqfinp
\end{equation}

\section{Problem formulation and resolution}\label{Schiever_problem_formulation_and_resolution}
In Sect.~\ref{Schiever:characteristics}, we have presented the hydrogen infrastructure and given its mathematical modeling. We now turn to formulate an optimization problem corresponding to the management of this infrastructure at minimum cost. In~\S\ref{Schiever_problem_formulation}, we give the problem formulation and, in~\S\ref{Schiever:price_decompo_method}, we propose a resolution method based on price decomposition.

\subsection{Problem formulation}\label{Schiever_problem_formulation}

We consider a probability space  $(\Omega,\mathcal{F},\mathds{P})$. Mathematical expectation is denoted by~$\mathbb{E}$. Random variables are denoted by bold capital letters like~$\va{Z}$.
The $\sigma$-field generated by~$\va{Z}$ is denoted by $\sigma(\va{Z})$. This notation is used to represent nonanticipativity constraints. All the variables introduced in Sect.\ref{Schiever:characteristics} are now random variables, hence represented by bold letters. We assume that $(\va{\Demand}_{\hour+1},\va{\Electricity}^{\PV}_{\hour+1})$ has a given probability distribution with finite support for all $\hour \in \HOUR$.

Gathering all that has been done in Sect.\ref{Schiever:characteristics}, we formulate the following minimization problem
\begin{subequations}\label{eq:Schiever_stochastic_formulation}
  \begin{equation}\label{eq:problem-formulation-stochastic}
    \begin{split}
      \min\limits_{
      \substack{
      \np{
      \va{\Control}_\hour}_{\hour \in \HOUR}
      \\ \np{\va{\State}_\hour }_{\hour \in \overline{\HOUR}}}}
  &
    \mathbb{E} \Big[ 
    \underbrace{\sum\limits_{\hour \in \HOUR} c^d  (\va{\Demand}_{\hour+1} - \va{\Hydrogen}^{\InDemand}_\hour )_{+}}_{\text{backup  cost}}
    +
    \underbrace{\sum\limits_{\hour \in \HOUR} c^{\PPA}  \va{\Electricity}^{\PPA}_\hour+ c_{\hour}^{G}  (\va{\Electricity}^{\Grid}_{\hour+1})_+
    + \FinalCost(\va{\ElectricityCumul}_\horizon)}_{\text{electricity cost}}
    \Big]
    \end{split}    
  \end{equation}

   subject to the following constraints for all $\hour \in \HOUR$
  \begin{align}
 &\substack{\text{operational}\\ \text{constraints}} \begin{cases} \label{eq:Schiever_stoch_operational_constraint}
 &\vaHydrogenProduced_\hour = \vaload_\hour\MaximalProduction(\vaMode_\hour,\vaTurnElectrolyser_\hour)\MaximalProductionHour   \eqsepv \\
  &\va{\Electricity}^{\Electrolyser}_\hour = \ELECTROLYSERFUNCTION(\vaload_\hour) \vaHydrogenProduced_\hour   + e^{\IDLE} \findi{\IDLE}  (\vaTurnElectrolyser_\hour)\mu(\vaMode_\hour,\vaTurnElectrolyser_\hour) \eqfinv
    \\
  &\va{\Electricity}^{\Compressor}_\hour = e^\Compressor\vaHydrogenProduced_\hour    \eqfinv
   \\
   &\vaload_\hour \in \LoadSetFunction(\vaTurnElectrolyser_\hour)      \eqfinv
\\
    & \Stock_0 \text{ given } \eqsepv \va{\Stock}_{\hour +1} 
    = \va{\Stock}_{\hour} + \vaHydrogenProduced_\hour  - \min(\va{\Demand}_{\hour+1} , \va{\Hydrogen}^{\InDemand}_\hour)  \eqfinv
    \\
      &\Mode_0 \text{ given } \eqsepv \vaMode_{\hour+1} =
     \vaTurnElectrolyser_\hour  \eqfinv
    \\
   &\underline{\Stock} \leq \va{\Stock}_\hour \leq \overline{\Stock} \eqfinv
    \\
    &0 
    \leq 
      \va{\Hydrogen}^{\InDemand}_\hour\eqfinv
    \end{cases}
    \end{align}
    \begin{align}
    &\hspace{-1.4cm}\substack{\text{electricity}\\ \text{constraints}}\begin{cases} \label{eq:Schiever_stoch_electricity_constraint}
        &\PPAStock_0=\overline{\Electricity^{\PPA}} \eqsepv \va{\PPAStock}_{\hour+1} = \va{\PPAStock}_\hour - \va{\Electricity}^{\PPA}_\hour  \eqfinv 
    \\
    &\ElectricityCumul_0=0 \eqsepv \va{\ElectricityCumul}_{\hour+1} =\va{\ElectricityCumul}_\hour +(1-p)(\va{\Electricity}_{\hour+1}^{\Grid})_+ \\ &\hspace{3cm} -p\min(\overline{\Electricity},\va{\Electricity}^{\PPA}_\hour+\va{\Electricity}_{\hour+1}^{\PV}) \eqfinv 
    \\
   &\va{\Electricity}^{\PPA}_{\hour} \leq \va{\PPAStock}_\hour \eqfinv 
    \\
    &0 \leq \va{\Electricity}^{\PPA}_\hour \eqfinv
        \\
        &\va{\Electricity}^{\PPA}_\hour + \va{\Electricity}^{\Grid}_{\hour+1}+ \va{\Electricity}^{\PV}_{\hour+1} \leq \overline{\Electricity} \eqfinv
        \\
        &\underline{\Electricity^{\Grid}} \leq \va{\Electricity}^{\Grid}_{\hour+1} \eqfinv 
      \end{cases}
    \end{align}
    \begin{align}
     &\hspace{-3.1cm}\substack{\text{coupling}\\ \text{constraint}}\begin{cases} 
     &\va{\Electricity}^{\PPA}_\hour + \va{\Electricity}^{\Grid}_{\hour+1} +  \va{\Electricity}^{\PV}_{\hour+1}
     = \va{\Electricity}^{\Electrolyser}_\hour + \va{\Electricity}^{\Compressor}_\hour \label{eq:Schiever_stoch_coupling_constraint}  \eqfinv
   \end{cases}
    \end{align}
    \begin{align}\label{eq:schiever_nonancitipativity}
      &\hspace{-0.6cm}\substack{\text{nonanticipativity}\\ \text{constraints}} \begin{cases}
          \sigma(\va{\Electricity}^{\PPA}_\hour,
  \vaload_\hour, 
  \vaTurnElectrolyser_\hour,
  \va{\Hydrogen}^{\InDemand}_\hour)
  &\subset \sigma\bp{(\va{\Demand}_{\hour '},\va{\Electricity}^{\PV}_{\hour'}) \eqsepv  \hour' \le \hour} \eqfinv
  \\
    \sigma(\va{\Electricity}^{\Grid}_{\hour+1})
  &\subset \sigma\bp{(\va{\Demand}_{\hour '},\va{\Electricity}^{\PV}_{\hour'}) \eqsepv  \hour' \le \hour+1} \eqfinv
\end{cases}
    \end{align}
\end{subequations}
where for every timestep (hour) $\hour \in \overline{\HOUR}$, the state vector $\va{\State}_\hour$ is defined by 
\begin{equation}\label{eq:Schiever_state}
  \va{\State}_\hour = \Bp{ 
    \va{\Stock}_\hour,
    \vaMode_\hour,
    \va{\PPAStock}_\hour, 
    \va{\ElectricityCumul}_\hour
  } \eqfinp
\end{equation}
The state at initial time, $\State_0$, is deterministic. Now, we comment the different blocks of constraints.
\begin{description}
\item[Operational constraints] The block
  constraints~\eqref{eq:Schiever_stoch_operational_constraint} describes the
  constraints related to the electrolyser, compressor and storage, as outlined in Equations~\eqref{eq:elementary_mode_dynamics1}, \eqref{eq:hydrogen_production}, \eqref{eq:electrolyser_electricity}, \eqref{eq:compressor_electricity}, \eqref{eq:mode_dynamics}, \eqref{eq:stock_dynamics} and~\eqref{eq:stock_bounded}.
\item[Electricity constraints] The block
  constraints~\eqref{eq:Schiever_stoch_electricity_constraint} describes the
  constraints related to the electricity sources, as discussed in Equations~\eqref{eq:cumul_dynamics}, \eqref{eq:schiever_total_elec_bounded}, \eqref{eq:ppa}, \eqref{eq:schiever_grid_bounded} and~\eqref{eq:PPA_dynamcis}.
\item[Coupling constraint] The
  constraints~\eqref{eq:Schiever_stoch_coupling_constraint} links the electricity
  consumption of the equipment and the electricity furnished by the three
  electricity sources, as~described in Equation~\eqref{eq:schiever-elec-division}.
\item[Nonanticipativity constraints] The decisions
  $(\va{\Electricity}^{\PPA}_\hour, \vaload_\hour, \vaTurnElectrolyser_\hour,
  \va{\Hydrogen}^{\InDemand}_\hour)$ at hour $\hour$ are taken knowing the
  uncertainties up to hour $\hour$, which can be written as the first constraint
  of~\eqref{eq:schiever_nonancitipativity}.  Moreover,
  $\va{\Electricity}^{\PV}_{\hour+1}$ is observed at the end of hour $\hour$; therefore, we
  require a recourse action to ensure the validity of
  constraint~\eqref{eq:Schiever_stoch_coupling_constraint}. For that reason, the
  decision $\va{\Electricity}^{\Grid}_{\hour+1}$ is taken knowing the
  uncertainties up to $\hour+1$, which can be written as the second constraint
  of~\eqref{eq:schiever_nonancitipativity}.
\end{description}

Note that the constraints defined in Problem~\eqref{eq:Schiever_stochastic_formulation} are almost sure constraints, that means they hold for $\mathds{P}$-almost all realizations of the random vector $(\va{\Demand}_{\hour+1},\va{\Electricity}^{\PV}_{\hour+1})_{\hour \in \HOUR}$  ($\mathds{P}$-\as).
The solutions of Problem~\eqref{eq:Schiever_stochastic_formulation} are sequences of hourly policies, that return the optimal decision for each hour $\hour \in \HOUR$ given the current state of the infrastructure $\State_\hour$.

\subsection{Resolution using price decomposition}\label{Schiever:price_decompo_method}
When the random variables $(\va{\Demand}_{\hour+1},\va{\Electricity}^{\PV}_{\hour+1})_{\hour \in \HOUR}$  are stagewise independent, Dynamic Programming provides an optimal solution to Problem~\eqref{eq:Schiever_stochastic_formulation}. Anyway, without stagewise independence, Dynamic Programming can be used to obtain admissible solution. However, 
Solving Problem~\eqref{eq:Schiever_stochastic_formulation} using Dynamic Programming is numerically difficult for the following reasons:  one week horizon with hourly decisions gives $168$ timesteps; a four dimensional state (see Equation~\eqref{eq:Schiever_state}) where the PPA-stock ($\PPAStock$)
and the cumulated electricity ($\ElectricityCumul)$ take values in large
interval and require a fine discretization which is numerically demanding; five decisions at each hour (see Equation~\eqref{eq:Schiever_decision}) have to be taken into account in the optimization algorithm.

\subsubsection{Sketch of the method}
As solving Problem~\eqref{eq:Schiever_stochastic_formulation} using Dynamic
Programming is numerically difficult, we propose an original decomposition
method in order to improve numerical tractability with the following steps.

\begin{enumerate}
\item We use \emph{Lagrangian relaxation} of coupling
  constraints~\eqref{eq:Schiever_stoch_coupling_constraint} to obtain an
  additive dual function
  $\phi[\FinalCost]: \lambda \in \RR^{\horizon} \mapsto \phi^O(\lambda) +\phi^E[\FinalCost](\lambda)$, where $\FinalCost$ is the final cost defined in Equation~\eqref{eq:heavyside}.  Denoting by
  $\mathrm{val}(\mathcal{D}[K])$ the value of the associated dual problem, that is
  $\mathrm{val}(\mathcal{D}[K]) = \sup_{\lambda \in \RR^{\horizon}} \phi[\FinalCost](\lambda)$, we
  obtain by weak duality that
  $\mathrm{val}(\mathcal{D}[\FinalCost]) \leq \mathrm{val}(\mathcal{P}[\FinalCost])$, where
  $\mathrm{val}(\mathcal{P}[\FinalCost])$ is the value of
  Problem~\eqref{eq:Schiever_stochastic_formulation}.
    
\item
  We make a detour by considering a new additive function
  $\widehat{\phi}[\widehat{\FinalCost}]: \lambda \in \RR^{\horizon} \mapsto
  \phi^O(\lambda)+\widehat{\phi}^E[\widehat{\FinalCost}](\lambda)$, where
  $\widehat{\FinalCost}$ is a nondecreasing convex proper function, and
  where, for each value of $\lambda \in \RR^{\horizon}$,
  $\widehat{\phi}^E[\widehat{\FinalCost}](\lambda)$ is the value of a convex optimization
  problem which is equivalent (in the sense that the value of the two problems
  coincide and solutions to either problem can be derived from one another) to
  the optimization problem whose value is
  $\phi^E[\widehat{\FinalCost}](\lambda)$.  Moreover, we prove in
  Proposition~\ref{Schiever_proposition} that $\widehat{\FinalCost}$ can be
  chosen in such a way that
  $\mathrm{val}(\mathcal{D}[\widehat{\FinalCost}]) \leq \mathrm{val}(\mathcal{D}[\FinalCost])$.
  
\item  \label{step3}
  We have that $\sup_{\lambda \in \RR^{\horizon}} \widehat{\phi}[\widehat{\FinalCost}](\lambda)
  = \sup_{\lambda \in \RR^{\horizon}} \phi[\widehat{\FinalCost}](\lambda) = \mathrm{val}(\mathcal{D}[\widehat{K}])$
  where $\mathrm{val}(\mathcal{D}[\widehat{K}])$ is the value of the Lagrangian dual (with respect to the coupling
  constraint~\eqref{eq:Schiever_stoch_coupling_constraint})
  of Problem~\eqref{eq:Schiever_stochastic_formulation}
  where the final cost $\FinalCost$ is replaced by $\widehat{\FinalCost}$.
  We numerically maximize the new function $\widehat{\phi}[\widehat{\FinalCost}]$.
\item For each value of $\lambda$, we can build an admissible 
  policy $\pi^{\lambda}$ for the original Problem~\eqref{eq:Schiever_stochastic_formulation} and
  we can obtain by Monte-Carlo simulation an approximation of the cost associated to that
  policy denoted by $\mathrm{val}(\mathcal{P_{\pi^{\lambda}}}[K])$ which gives an upper bound of
  $\mathrm{val}(\mathcal{P}[K])$ the value of Problem~\eqref{eq:Schiever_stochastic_formulation}.
  We use this fact to simulate an  admissible 
  policy associated to the best $\lambda$ obtained at the previous step~\ref{step3}. 
\end{enumerate}

Summarizing the previous steps, we have
\begin{equation}\label{eq:comparison_dual_primal_stochastic}
    \lefteqn{\underbrace{\phantom{\mathrm{val}(\mathcal{D}[\widehat{\FinalCost}]) \leq \mathrm{val}(\mathcal{D}[\FinalCost]) }}_{\text{using Proposition~\ref{Schiever_proposition}}}}     \lefteqn{\mathrm{val}(\mathcal{D}[\widehat{\FinalCost}]) \leq \overbrace{\phantom{\mathrm{val}(\mathcal{D}[\FinalCost])  \leq \mathrm{val}(\mathcal{P}[\FinalCost])  }}^{\text{Weak duality}}}    \mathrm{val}(\mathcal{D}[\widehat{\FinalCost}]) \leq \mathrm{val}(\mathcal{D}[\FinalCost]) \leq \underbrace{\mathrm{val}(\mathcal{P}[\FinalCost]) \leq \mathrm{val}(\mathcal{P_{\pi^{\lambda}}}[\FinalCost])}_{\text{Feasibility of policy $\pi^\lambda $}} \eqfinv
  \end{equation}
where the final cost function $\widehat{\FinalCost}$ is the one described
  at item~\ref{proposition_item3} of Proposition~\ref{Schiever_proposition}.

The duality gap of Problem~\eqref{eq:Schiever_stochastic_formulation}, which is
defined by
$\mathrm{val}(\mathcal{P}[\FinalCost])- \mathrm{val}(\mathcal{D}[\FinalCost])$, is numerically
intractable as it requires maximizing the dual function
$\phi[\FinalCost]$. However, by using
Equation~\eqref{eq:comparison_dual_primal_stochastic}, we can bound the duality
gap
$\mathrm{val}(\mathcal{P}[\FinalCost])- \mathrm{val}(\mathcal{D}[\FinalCost])$ by
$\mathrm{val}(\mathcal{D}[\widehat{\FinalCost}]) -
\mathrm{val}(\mathcal{P_{\pi^{\lambda}}}[\FinalCost])$, which is numerically tractable.

\subsubsection{Relaxation with deterministic Lagrange multiplier}

We observe that Problem~\eqref{eq:Schiever_stochastic_formulation} is the
minimum of the sum of a backup cost and an electricity cost with two different
blocks of constraints~\eqref{eq:Schiever_stoch_operational_constraint} and
\eqref{eq:Schiever_stoch_electricity_constraint} (each block having its own
variables) and one coupling
constraint~\eqref{eq:Schiever_stoch_coupling_constraint} for all
$\hour \in \HOUR$.

 We consider a decomposition algorithm by dualizing the coupling
 constraint~\eqref{eq:Schiever_stoch_coupling_constraint}.
As constraint~\eqref{eq:Schiever_stoch_coupling_constraint} is stochastic, it is natural to dualize with stochastic Lagrange multipliers. However,
the optimization over stochastic Lagrange multipliers presents intractability
challenges. Therefore, we only consider deterministic Lagrange multipliers. Indeed, we will observe that weak duality is enough to obtain good numerical bounds. We
recall that maximizing the stochastic dual function over the restricted set of
deterministic multipliers leads to a lower bound of the optimal value of the original problem. The decomposition method is presented now.

Given a deterministic multiplier $\lambda = (\lambda_\hour)_{\hour \in \HOUR}
\in \mathbb{R}^\horizon$, we denote by $\phi[\FinalCost](\lambda)$ the dual
function associated with the final cost $\FinalCost$ of
Problem~\eqref{eq:Schiever_stochastic_formulation}
\begin{align}
  \phi[\FinalCost](\lambda) =
  &\min\limits_{
    \substack{\np{ 
    \va{\Control}_\hour}_{\hour \in \HOUR} \\\np{\va{\State}_\hour }_{\hour \in \overline{\HOUR}}}}
  \mathbb{E}  \Big[
  \sum\limits_{\hour \in \HOUR} c^d  (\va{\Demand}_{\hour+1} - \va{\Hydrogen}^{\InDemand}_\hour )_{+}
  + \sum\limits_{\hour \in \HOUR} c^{\PPA}  \va{\Electricity}^{\PPA}_\hour+ c_{\hour}^{g} (\va{\Electricity}^{\Grid}_{\hour+1})_+
  \nonumber
  \\ 
  &\hspace{0.5cm}+\FinalCost(\va{\ElectricityCumul}_\horizon) +
    \sum\limits_{\hour \in \HOUR} \lambda_\hour(-\va{\Electricity}^{\PPA}_\hour -
    \va{\Electricity}^{\Grid}_{\hour+1} -  \va{\Electricity}^{\PV}_{\hour+1}
    +\va{\Electricity}^{\Electrolyser}_\hour + \va{\Electricity}^{\Compressor}_\hour) \Big] 
    \label{eq:schiever_dual_function}
  \\ 
  &\hspace{0.5cm}\text{s.t.~\eqref{eq:Schiever_stoch_operational_constraint}, \eqref{eq:Schiever_stoch_electricity_constraint}, \eqref{eq:schiever_nonancitipativity}}
    \eqfinp \nonumber
\end{align}

Note that the final cost $\FinalCost$ is put as a parameter of the dual function
$\phi$ for future use and, we denote by Problem~$[\FinalCost]$-\eqref{eq:schiever_dual_function} the Problem~\eqref{eq:schiever_dual_function} where the final cost $\FinalCost$ is considered.

By weak duality, the dual
function $\phi[\FinalCost](\lambda)$ is a lower bound of the value of
Problem~\eqref{eq:Schiever_stochastic_formulation} for all deterministic
multiplier $\lambda \in \mathbb{R}^\horizon$
\begin{equation}\label{eq:Schiever_stochastic_weak_duality}
  \phi[\FinalCost](\lambda) \leq \mathrm{val}(\mathcal{P}[\FinalCost]) \eqsepv \forall \lambda \in \mathbb{R}^\horizon \eqfinp
\end{equation}
We rewrite the dual function $\phi[\FinalCost]$ as a sum 
\begin{subequations}
  \begin{equation}
    \phi[\FinalCost](\lambda) = \phi^{O}(\lambda) + \phi^{\EL}[\FinalCost](\lambda) \eqfinv \label{eq:Schiever_stochastic_dual_function}
  \end{equation}
  where the function $\phi^{O}$, which represents what we call the \emph{operational problem}, is defined for all $\lambda \in \mathbb{R^{\horizon}}$ by 
  \begin{align}
    \phi^{O}(\lambda)
    &=
      \min_{
      \substack{
      \np{\vaTurnElectrolyser_\hour, \vaload_\hour, \va{\Hydrogen}^{\InDemand}_\hour}_{\hour \in \HOUR}
    \\
    \np{\va{\Stock}_\hour,\vaMode_\hour}_{\hour \in \overline{\HOUR}}}}
    \mathbb{E} \Big[  \sum_{\hour \in \HOUR} c^d  (\va{\Demand}_{\hour+1} - \va{\Hydrogen}^{\InDemand}_\hour )_{+}
    + \sum_{\hour \in \HOUR}\lambda_\hour(\va{\Electricity^{\Electrolyser}_\hour} + \va{\Electricity^{\Compressor}_\hour} ) \Big]
    \label{eq:Schiever_stochastic_decompo_operational}  \\
    &\hspace{2cm}\text{s.t.~\eqref{eq:Schiever_stoch_operational_constraint} \nonumber \eqfinv}
    \\
    &\hspace{2cm}\text{and }\sigma(\vaload_\hour, \vaTurnElectrolyser_\hour, \va{\Hydrogen}^{\InDemand}_\hour)
      \subset \sigma\bp{\va{\Demand}_{\hour '} \eqsepv \hour' \le \hour} \eqsepv \forall \hour \in \HOUR \eqfinv
      \nonumber
  \end{align}
  and the function $\phi^{\EL}$, which represents what we call the \emph{electricity allocation problem}, is defined for all $\lambda \in \mathbb{R^{\horizon}}$ by
  \begin{align}
    \phi^{\EL}[\FinalCost](\lambda)
    &= \min_{
      \substack{
      \np{\va{\Electricity}^{\PPA}_\hour,  \va{\Electricity}^{\Grid}_{\hour+1}}_{\hour \in \HOUR}\\
    \np{\va{\PPAStock}_{\hour},\va{\ElectricityCumul}_{\hour}}_{\hour \in \overline{\HOUR}}}}
    \mathbb{E} \Big[\sum_{\hour \in \HOUR}  \InstantaneousCost_\hour^{\EL}\bp{\va{\Electricity}^{\PPA}_\hour, \va{\Electricity}^{\Grid}_{\hour+1},\va{\Electricity}^{\PV}_{\hour+1},\lambda_\hour} + \FinalCost(\va{\ElectricityCumul}_\horizon) \Big] \label{eq:Schiever_stochastic_decompo_electricity} \\
 &\hspace{3cm}\text{s.t. ~\eqref{eq:Schiever_stoch_electricity_constraint} \eqfinv \nonumber} 
  \\
  &\hspace{3cm}\text{and }\sigma(\va{\Electricity}^{\PPA}_\hour) \subset \sigma\bp{\va{\Electricity}^{\PV}_{\hour '} \eqsepv \hour' \le \hour} \eqsepv \forall \hour \in \HOUR \eqfinv \nonumber \\
    &\hspace{3cm}\sigma(\va{\Electricity}^{\Grid}_{\hour+1}) \subset \sigma\bp{\va{\Electricity}^{\PV}_{\hour '} \eqsepv  \hour' \le \hour+1} \eqsepv \forall \hour \in \HOUR \eqfinv
                     \nonumber
  \end{align}
  where for all $\hour \in \HOUR$
\begin{equation}
  \InstantaneousCost_\hour^{\EL}\bp{\va{\Electricity}^{\PPA}_\hour,
    \va{\Electricity}^{\Grid}_{\hour+1},\va{\Electricity}^{\PV}_{\hour+1}, \lambda_\hour} = c^{\PPA}  \va{\Electricity}^{\PPA}_\hour+ c_{\hour}^{G} (\va{\Electricity}^{\Grid}_{\hour+1})_+  -
  \lambda_\hour(\va{\Electricity}^{\PPA}_\hour + \va{\Electricity}^{\Grid}_{\hour+1} +  \va{\Electricity}^{\PV}_{\hour+1}) \eqfinp
\end{equation}
\end{subequations}

Given $\lambda \in \mathbb{R}^\horizon$, each
subproblem~\eqref{eq:Schiever_stochastic_decompo_operational} and
\eqref{eq:Schiever_stochastic_decompo_electricity} can be solved independently.

\subsubsection{An equivalent convex electricity allocation problem}\label{Schiever_auxiliary_electricity}
The electricity allocation Problem~$[\FinalCost]$-\eqref{eq:Schiever_stochastic_decompo_electricity} obtained by decomposing Problem~\mbox{$[\FinalCost]${-}\eqref{eq:Schiever_stochastic_formulation}} is still a
challenge for Dynamic Programming as it requires a fine discretization of
  the two states~$\PPAStock$ and $\ElectricityCumul$. An approximation by a convex
optimization problem would enable the use of faster algorithms for its
resolution, like Stochastic Dual Dynamic Programming (SDDP) \cite{SDDP}, which does not rely on state discretization, and is particularly adapted to the stochastic case.

To obtain a convex approximation of
Problem~$[\FinalCost]$-\eqref{eq:Schiever_stochastic_decompo_electricity}, we use the two following
keys. First, we substitute the nonconvex final cost function $\FinalCost$ with
a proper nondecreasing convex function $\widehat{\FinalCost}$.
Second, we replace the cumulative electricity $\ElectricityCumul$
dynamics (nonlinear) as described in
Equation~\eqref{eq:cumul_dynamics} with a linear
dynamics by introducing new decisions and constraints. The new optimization
problem we consider is defined by

\begin{subequations}\label{eq:Schiever_stochastic_decompo_electricity2}
  \begin{equation}
    \begin{split}
      \widehat{\phi}^{\EL}[\widehat{\FinalCost}](\lambda) =
      \min_{
      \substack{
      \np{ 
      \va{\Electricity}^{\PPA}_\hour, 
      \va{\Electricity}^{\Grid}_{\hour+1},
      \va{\Electricity}^N_{\hour+1}, 
      \va{\Electricity}^R_{\hour+1}}_{\hour \in \HOUR}
      \\
      \np{\va{\PPAStock}_\hour,
      \va{\ElectricityCumul}_\hour
      }_{\hour \in \overline{\HOUR}}}
      }
      &\mathbb{E} \Big[\sum\limits_{\hour \in \HOUR} \InstantaneousCost_\hour^{\EL}\bp{\va{\Electricity}^{\PPA}_\hour,
        \va{\Electricity}^{\Grid}_{\hour+1},\va{\Electricity}^{\PV}_{\hour+1},\lambda_\hour} + \widehat{\FinalCost}(\va{\ElectricityCumul_\horizon})  \Big]  
    \end{split}
  \end{equation}

  subject to the following constraint, $\forall \hour \in \HOUR$
  \begin{align}
    \substack{\text{reformulation of~\eqref{eq:Schiever_stoch_electricity_constraint}} \\ \text{as linear constraints}
   }
    &\begin{cases}
    &\PPAStock_0=\overline{\Electricity^{\PPA}} \eqsepv \va{\PPAStock}_{\hour+1} = \va{\PPAStock}_\hour - \va{\Electricity}^{\PPA}_\hour  \eqfinv
    \\
    &\ElectricityCumul_0=0 \eqsepv \va{\ElectricityCumul}_{\hour+1} =\va{\ElectricityCumul}_\hour +(1-p)\va{\Electricity}_{\hour+1}^{N} -p\va{\Electricity}^{R}_{\hour+1}  \eqfinv
    \\
    &\va{\Electricity}^{\PPA}_{\hour} \leq \va{\PPAStock}_\hour \eqfinv
    \\
    &0 \leq \va{\Electricity}^{\PPA}_\hour \eqfinv
    \\
    &\va{\Electricity}^{\PPA}_\hour + \va{\Electricity}^{\Grid}_{\hour+1}+ \va{\Electricity}^{\PV}_{\hour+1} \leq \overline{\Electricity} \eqfinv
    \\
    &\underline{\Electricity}^{\Grid} \leq \va{\Electricity}^{\Grid}_{\hour+1} \eqfinv
  \end{cases}\label{eq:schiever_stochastic_initial_constraints}
  \\
  \substack{\text{additional constraints}}&\begin{cases} &0 \leq \va{\Electricity}_{\hour+1}^{N} \eqfinv 
    \\ 
    &\va{\Electricity}_{\hour+1}^{\Grid} \leq \va{\Electricity}_{\hour+1}^{N} \eqfinv 
    \\ 
    &\va{\Electricity}^{R}_{\hour+1} \leq \overline{\Electricity} \eqfinv 
    \\
    &\va{\Electricity}^{R}_{\hour+1} \leq \va{\Electricity}^{\PPA}_\hour + \va{\Electricity}^{\PV}_{\hour+1} \eqfinv
  \end{cases}\label{eq:schiever_stochastic_linearization_constraints}
  \\
  \substack{\text{nonanticipativity} \\\text{constraints}}&\begin{cases}
    &\sigma(\va{\Electricity}^{\PPA}_\hour)
    \subset \sigma\bp{\va{\Electricity}^{\PV}_{\hour '} \eqsepv \hour' \le \hour} \eqfinv
    \\ 
    &\sigma (\va{\Electricity}^{\Grid}_{\hour+1},\va{\Electricity}^{N}_{\hour+1},\va{\Electricity}^{R}_{\hour+1})
    \subset \sigma\bp{\va{\Electricity}^{\PV}_{\hour '} \eqsepv \hour' \le \hour+1} \eqfinp
  \end{cases}\label{eq:schiever_stochastic_nonanticipativity_constraints}
  \end{align}
\end{subequations}
In~\eqref{eq:schiever_stochastic_linearization_constraints},
the constraints on $\va{\Electricity}^{N}_{\hour+1}$ model the positive part of $\va{\Electricity}^{\Grid}_{\hour+1}$ in the dynamics of $\va{\ElectricityCumul}$. The constraints on $\va{\Electricity}^{R}_{\hour+1}$ model the min function in the dynamics of $\va{\ElectricityCumul}$.

In Proposition~\ref{Schiever_proposition}, we show that Problem~$[\widehat{\FinalCost}]$-\eqref{eq:Schiever_stochastic_decompo_electricity} and
Problem~$[\widehat{\FinalCost}]$-\eqref{eq:Schiever_stochastic_decompo_electricity2} when considered with the same final cost $\widehat{\FinalCost}$ are equivalent, in the sense that from a feasible solution of Problem~$[\widehat{\FinalCost}]$-\eqref{eq:Schiever_stochastic_decompo_electricity} (resp. Problem~$[\widehat{\FinalCost}]$~-\eqref{eq:Schiever_stochastic_decompo_electricity2}), we can construct a feasible solution for Problem~$[\widehat{\FinalCost}]$-\eqref{eq:Schiever_stochastic_decompo_electricity2} (resp. Problem~$[\widehat{\FinalCost}]$-\eqref{eq:Schiever_stochastic_decompo_electricity}) that yields the same value. Moreover, we give in Proposition~\ref{Schiever_proposition} conditions on the choice of $\widehat{\FinalCost}$ to obtain lower bounds on the value of Problem~\eqref{eq:Schiever_stochastic_decompo_electricity} with the original final cost ${\FinalCost}$.

\begin{proposition}\label{Schiever_proposition}
  We consider Problem~\eqref{eq:Schiever_stochastic_decompo_electricity} and
  Problem~\eqref{eq:Schiever_stochastic_decompo_electricity2}.
  \begin{enumerate}
  \item\label{proposition_item1} If 
    the final cost function $\overline{\FinalCost}$ in the definition of
    Problem~$[\overline{\FinalCost}]$-\eqref{eq:Schiever_stochastic_decompo_electricity2}
    is proper\footnote{that is $\overline{\FinalCost}: \RR \to ]-\infty,+\infty]$ and there exists $x \in \RR$ such that
      $\overline{K}(x) \ne +\infty$} and nondecreasing, then
    Problem~$[\overline{\FinalCost}]$-\eqref{eq:Schiever_stochastic_decompo_electricity2} is equivalent to
    Problem~$[\overline{\FinalCost}]$-\eqref{eq:Schiever_stochastic_decompo_electricity}.
    Moreover, if
    $\overline{\FinalCost}$ is convex, then
    Problem~$[\overline{\FinalCost}]$-\eqref{eq:Schiever_stochastic_decompo_electricity2} is a convex
    optimization problem.
  \item \label{proposition_item2}
    If $\overline{\FinalCost} \leq \FinalCost$ in the interval
    [$\underline{\ElectricityCumul},\overline{\ElectricityCumul}$], where $\underline{\ElectricityCumul}$ and $\overline{\ElectricityCumul}$ are defined
    respectively \\by $\overline{\ElectricityCumul}= \horizon (1-p)
    \overline{\Electricity}$ and $\underline{\ElectricityCumul} = -\horizon p
    \overline{\Electricity}$, then we have that
      $\phi^{\EL}[\overline{\FinalCost}] \le \phi^{\EL}[\FinalCost]$.
  \item \label{proposition_item3}
    The final cost
    $\widehat{\FinalCost}: \mathbb{R} \to \mathbb{R}$ is defined by $ \widehat{\FinalCost}(x) =
    \max(\beta_1 x, \beta_2 x) - c^s$   with $\beta_1,\beta_2$ such that
    $0 \leq \beta_1 < \beta_2 \leq \frac{c^s}{\overline{\ElectricityCumul}}$ and where $c^s$ is
    the subsidy, as described in~\S\ref{Schiever:cost_function} satisfies previous items~\ref{proposition_item1} and~\ref{proposition_item2}.
    As a consequence we have that
    $\widehat{\phi}[\widehat{\FinalCost}] \le \phi[{\FinalCost}]$
    where $\widehat{\phi}[\widehat{\FinalCost}] = \phi^{\OP} +\widehat{\phi}^{\EL}[\widehat{\FinalCost}]$.
  \end{enumerate}
\end{proposition}

\begin{proof}
  See Appendix~\ref{Schiever_appendix_lemma} for the proof of Proposition~\ref{Schiever_proposition}.
\end{proof}

\subsubsection{Maximizing the new additive function $\widehat{\phi}[\widehat{\FinalCost}] = \phi^{\OP} + \widehat{\phi}^{\EL}[\widehat{\FinalCost}]$}\label{maximizing_phi_lin}

In what follows, we assume that $\widehat{\FinalCost}$ satisfies the assumptions of item~\ref{proposition_item3} of Proposition~\ref{Schiever_proposition}, and thus, for any $\lambda \in \mathbb{R}^\horizon$,
$\widehat{\phi}[\widehat{\FinalCost}](\lambda)$ gives a lower bound of $\mathrm{val}(\mathcal{P}[K])$.
In order to obtain the best lower bound, 
we numerically maximize the function $\widehat{\phi}[\widehat{\FinalCost}]$
using an iterative gradient-like\footnote{that is a substitute of the gradient} based algorithm whose steps are now detailed.

\begin{description}
\item[Step 1:] {\bf Initialization of the Lagrange multiplier $\lambda^0$}

  In order to choose a good initial value for the Lagrange multiplier, we use a deterministic idealized problem (convex optimization problem) whose optimal solution
  satisfies certain conditions and for which we are able to find a lower bound
  for $\lambda$. This is done by applying Lemma~\ref{lemma:schiever_lemma_multiplier} for all $\hour \in \HOUR$, which gives us the following lower bounds
  \begin{equation}\label{eq:Schiever_lambda_0} \lambda^0_\hour =
    pc^{\Grid}_\hour + (1-p)c^{\PPA} \eqsepv \forall \hour \in \HOUR \eqfinp
  \end{equation}
  Note that $\lambda^{0}$ only depends on the parameters of the electricity allocation
  Problem~\eqref{eq:Schiever_stochastic_decompo_electricity}.  We use these lower bounds in our numerical experiments as a
  starting point to maximize the dual value function
  $\widehat{\phi}[\widehat{\FinalCost}]$. This initialization gives good results as
  displayed in Figure~\ref{fig:Schiever_dual_function_stochastic}.

\item[Step 2:] {\bf Gradient-like based maximization of the dual function $\widehat{\phi}[\widehat{\FinalCost}]$}
  
  Second, at each iteration $k$ of the algorithm, the gradient-like of the
  function $\widehat{\phi}[\widehat{\FinalCost}]$ at point $\lambda^k$ is computed using Equation~\eqref{eq:Schiever_stochastic_decomposition_gradient}. For that purpose, we need to compute
  the optimal decisions of the operational
  Problem~\eqref{eq:Schiever_stochastic_decompo_operational} and the electricity allocation
  Problem~\eqref{eq:Schiever_stochastic_decompo_electricity}, which is done  as
  follows.

  \begin{description}
  \item[Step 2.1:] {\bf Solving the operational problem $\phi^{\OP}$}

    The operational Problem~\eqref{eq:Schiever_stochastic_decompo_operational} is solved by Stochastic Dynamic Programming with the pair $(\Stock,\Mode)$ composed of the stock of hydrogen and the mode of the electrolyser as state variables. The Bellman value functions~\cite{Bellman} is given by the following induction. For all $\hour\in\HOUR,$ for all $\Stock_\hour,\Mode_\hour$
    \begin{align}\label{eq:Schiever_stoch_decompo_op_dp}
      \Value^{\OP,\lambda}_\hour(\Stock_\hour,\Mode_\hour)
      = 
      \min\limits_{\np{ 
      \TurnElectrolyser_\hour,
      \load_\hour,
      \Hydrogen^{\InDemand}_\hour
      }} 
      &\mathbb{E}_{\va{\Demand}_{\hour+1}} \Big[c^d  (\va{\Demand}_{\hour+1} - \Hydrogen^{\InDemand}_\hour )_{+} + \lambda_\hour(\Electricity^{\Electrolyser}_\hour +\Electricity^{\Compressor}_\hour)  \nonumber 
      \\
      &\hspace{1.2cm} + \Value^{\OP,\lambda}_{\hour+1}\big(\va{\Stock}_{\hour+1},\Mode_{\hour+1}\big) \Big]
      \\ 
      &\text{s.t.~\eqref{eq:Schiever_stoch_operational_constraint}} \eqfinv\nonumber
    \end{align}
    
   where the Bellman value function at time $\hour = \horizon$ is null. When the random variables $(\va{\Demand}_{\hour+1})_{\hour \in \HOUR}$  are stagewise independent, Dynamic Programming provides an optimal solution.
  \item[Step 2.2:] {\bf Solving the electricity allocation problem $\phi^{\EL}[\widehat{\FinalCost}]$}

    While Stochastic Dynamic Programming is applicable to the problem, its
    practical implementation is computationally intensive due to the need for
    precise discretization of the states $\PPAStock$ and
    $\ElectricityCumul$. Alternatively, Stochastic Dual Dynamic Programming
    (SDDP), leveraging the convex nature of the problem, offers a promising
    alternative way of obtaining a solution. Note that using the final cost $\widehat{\FinalCost}$ defined in
    Item~\ref{proposition_item3} of Proposition~\ref{Schiever_proposition} is
    preferred when using SDDP, given its polyhedral nature.

    The Bellman value functions associated with the electricity allocation
    Problem~\eqref{eq:Schiever_stochastic_decompo_electricity2} are given by the following induction.
    
    At time $\hour = \horizon$, we have $ \Value^{\EL,\lambda}_\horizon(\PPAStock_\horizon,\ElectricityCumul_\horizon) = \widehat{\FinalCost}(\ElectricityCumul_\horizon)$ for all $\PPAStock_\horizon,\ElectricityCumul_\horizon$

    and for all $\hour\in\HOUR$, for all $\PPAStock_\hour,\ElectricityCumul_\hour$
    \begin{align}
    \Value^{\EL,\lambda}_\hour(\PPAStock_\hour,\ElectricityCumul_\hour)
      = 
      \min\limits_{\Electricity^{\PPA}_\hour } \:
      & \mathbb{E}_{\va{\Electricity}^{\PV}_{\hour+1}} \big[\min\limits_{\np{ 
        \va{\Electricity}^{\Grid}_{\hour+1}, 
        \va{\Electricity}^{N}_{\hour+1}, 
        \va{\Electricity}^{R}_{\hour+1}
        }}  \InstantaneousCost_\hour^{\EL}\bp{\Electricity^{\PPA}_\hour,
        \va{\Electricity}^{\Grid}_{\hour+1},\va{\Electricity}^{\PV}_{\hour+1},\lambda_\hour} \nonumber
      \\
      &\hspace{1.3cm}+ \Value^{\EL,\lambda}_{\hour+1}\big(\PPAStock_{\hour+1},\va{\ElectricityCumul}_{\hour+1}  \big)  \big] \label{eq:Schiever_decompo_elec_stoch_dp}
      \\
      &\hspace{1cm}\text{s.t.~\eqref{eq:schiever_stochastic_initial_constraints}, \eqref{eq:schiever_stochastic_linearization_constraints}} \nonumber \eqfinp 
    \end{align}

Stochastic Dual Dynamic Programming provides lower bounds $(\underline{\Value}^{E,\lambda}_\hour)_{\hour \in \HOUR}$ for the Bellman functions given by the Equations~\eqref{eq:Schiever_decompo_elec_stoch_dp}.

\item[Step 2.3:] {\bf Computation of the gradient of the function $\widehat{\phi}[\widehat{\FinalCost}]$}
  
    If the function $\widehat{\phi}[\widehat{\FinalCost}]$ was differentiable, we would obtain that
    \begin{equation}\label{eq:Schiever_stochastic_decomposition_gradient}
     \frac{\partial \widehat{\phi}[\widehat{\FinalCost}]}{\partial \lambda_\hour}(\lambda)= \widehat{\nabla}_{\hour}=  \mathbb{E} \Big[
        \va{\Electricity}^{\Electrolyser}_\hour
        +\va{\Electricity}^{\Compressor}_\hour \Big]  +\mathbb{E} \Big[
        -\va{\Electricity}^{\PPA}_\hour - \va{\Electricity}^{\Grid}_{\hour+1} -
        \va{\Electricity}^{\PV}_{\hour+1} \Big] \eqfinv
    \end{equation}
    for all $\hour \in \HOUR$, where $\va{\Electricity}^{\cdot}_\hour$ is the optimal value of the
    control $\va{\Electricity}^{\cdot}_\hour$ of Problems~\eqref{eq:Schiever_stochastic_decompo_operational} and~\eqref{eq:Schiever_stochastic_decompo_electricity} that depend on $\lambda$. 
    Here, as the function $\widehat{\phi}[\widehat{\FinalCost}]$ is not differentiable (presence of integer controls in Problem~\eqref{eq:Schiever_stochastic_decompo_operational}), we use Equation~\eqref{eq:Schiever_stochastic_decomposition_gradient} as
    a gradient-like heuristic to update the multiplier $\lambda$ when maximizing the function $\widehat{\phi}[\widehat{\FinalCost}]$.
    The gradient-like is defined by the sum of two expectations. The second one, $\mathbb{E} \big[
        -\va{\Electricity}^{\PPA}_\hour - \va{\Electricity}^{\Grid}_{\hour+1} -
        \va{\Electricity}^{\PV}_{\hour+1} \big]$, is approximated using a
    Monte-Carlo method while the first one, $\mathbb{E}\big[
        \va{\Electricity}^{\Electrolyser}_\hour
        +\va{\Electricity}^{\Compressor}_\hour \big]$, is more efficiently computed using the discrete probability law of the state driven by the optimal policy (Fokker-Planck equation).
  \end{description}
\end{description}

The algorithm used to maximize $\widehat{\phi}[\widehat{\FinalCost}]$ using
Stochastic Dynamic Programming for the operational Problem~\eqref{eq:Schiever_stochastic_decompo_operational} and Stochastic Dual
Dynamic Programming for the electricity allocation Problem~\eqref{eq:Schiever_stochastic_decompo_electricity} is described in
Algorithm~\ref{alg:Schiever_stochastic_decompo_algo}.
\begin{algorithm}
\caption{Maximizing the function $\widehat{\phi}[\widehat{\FinalCost}]$ by gradient ascent}
\label{alg:Schiever_stochastic_decompo_algo}
\begin{algorithmic}
\Require $\lambda^0, nb\_iterations, \gamma, \Stock_0, \Mode_0, \PPAStock_0, \ElectricityCumul_0$

\hspace{-0.85cm}$k \gets 0$

\hspace{-0.85cm}\text{Initialize $\lambda^0$ using Equation~\eqref{eq:Schiever_lambda_0}}

\While{$k < nb\_iterations$}

\text{Run SDP on operational problem to obtain $(\Value^{\OP,\lambda^{k-1}}_\hour)_{\hour \in \HOUR}$ \eqref{eq:Schiever_stoch_decompo_op_dp}}

\text{Run SDDP on the electricity problem to obtain $(\underline{\Value}^{\EL,\lambda^{k-1}}_\hour)_{\hour \in \HOUR}$ \eqref{eq:Schiever_decompo_elec_stoch_dp}}

 $\lambda^{k} \gets \lambda^{k-1} + \gamma (\widehat{\nabla}_\hour)_{\hour \in \HOUR}$ (see~\eqref{eq:Schiever_stochastic_decomposition_gradient})

$k \gets k+1$
\EndWhile

\hspace{-0.85cm}\Return $\lambda^{\text{nb\_iterations}}$
\end{algorithmic}
\end{algorithm}

\subsubsection{Producing an admissible policy} 
A (state) policy is a mapping from states to controls that determines the action to take at a given time in a given state.

For a fixed deterministic multiplier $\lambda = \sequence{\lambda_\hour}{\hour \in \HOUR}$, we obtain a feasible policy $\pi^{\lambda} = \sequence{\pi^{\lambda}_\hour}{\hour \in \HOUR}$ for Problem~\eqref{eq:Schiever_stochastic_formulation} by considering the following one step optimization problem which uses the sum of the computed Bellman value functions~\eqref{eq:Schiever_stoch_decompo_op_dp} for the operational problem and the lower bounds of the Bellman value functions~\eqref{eq:Schiever_decompo_elec_stoch_dp} for the electricity allocation problem
\begin{align}\label{eq:admissible-policy-stoch}
  \pi^{\lambda}_\hour(\Stock_\hour,\Mode_\hour,\PPAStock_\hour,\ElectricityCumul_\hour)
  &= \argmin\limits_{\np{\Electricity^{\PPA}_\hour,\TurnElectrolyser_\hour,\load_\hour,
  \Hydrogen^{\InDemand}_\hour}}  \:     
  \mathbb{E}  \Big[
  \min_{\va{\Electricity}^{\Grid}_{\hour+1}}\ \InstantaneousCost_\hour\bp{\Electricity^{\PPA}_\hour,
    \va{\Electricity}^{\Grid}_{\hour+1}, \va{\Demand}_{\hour+1},
    \Hydrogen^{\InDemand}_\hour}
    \nonumber \\
  &\hspace{1cm} \underbrace{\Value^{\OP,\lambda}_{\hour+1}(\va{\Stock}_{\hour+1}, \Mode_{\hour+1}) + \underline{\Value}^{\EL,\lambda}_{\hour+1}(\PPAStock_{\hour+1}, \va{\ElectricityCumul}_{\hour+1})}_{\text{surrogate additive value function}} \Big]  
  \\
  &\hspace{1cm}\text{s.t.~\eqref{eq:Schiever_stoch_operational_constraint}, \eqref{eq:Schiever_stoch_electricity_constraint}, \eqref{eq:Schiever_stoch_coupling_constraint}} \nonumber  \eqfinp
\end{align}
 We denote by $\mathrm{val}(\mathcal{P_{\pi^{\lambda}}}[\FinalCost])$ the total cost of Problem~\eqref{eq:Schiever_stochastic_formulation} when applying the feasible policy given by Equation~\eqref{eq:admissible-policy-stoch}.

\section{Numerical case study results}
\label{Schiever:Numerical_results}

In this section, we present numerical results obtained for Problem~\eqref{eq:Schiever_stochastic_formulation} described in~\S\ref{Schiever_problem_formulation}.

\subsection{Case study data}\label{shriver-data}
We present the different data needed to formulate Problem~\eqref{eq:Schiever_stochastic_formulation}. Some of the data were already presented in Table~\ref{table:cost-variables}. The optimization problem is formulated at hourly step over one week, thus we have $\HOUR=\{0,\ldots,167\}$.

\begin{itemize}
\item The electrolyser has the following characteristics.
  \begin{itemize} 
  \item The electrolyser function $\ELECTROLYSERFUNCTION$ (used in Equation~\eqref{eq:electrolyser_electricity}) is given in Figure~\ref{fig:schiever_load}.
    
  \item Table~\ref{table:Schiever_transition_time_electrolyser}
    gives the numerical values of the function $\MaximalProduction$ (see Equation~\eqref{eq:hydrogen_production}).
    
  \item The maximal hydrogen production $\MaximalProductionHour$ (used in Equation~\eqref{eq:hydrogen_production}) is equal to \numprint[kg/h]{23}. 
  \item The consumption on \IDLE\ mode $e^{\IDLE}$ (used in Equation~\eqref{eq:electrolyser_electricity}) is equal to \numprint[kWh]{3} per hour.
  \end{itemize}

  \begin{figure}[htpp]
    \centering
    \begin{subfigure}[h]{0.3\textwidth}
      \centering
      \mbox{\includegraphics[width=1.0\textwidth]{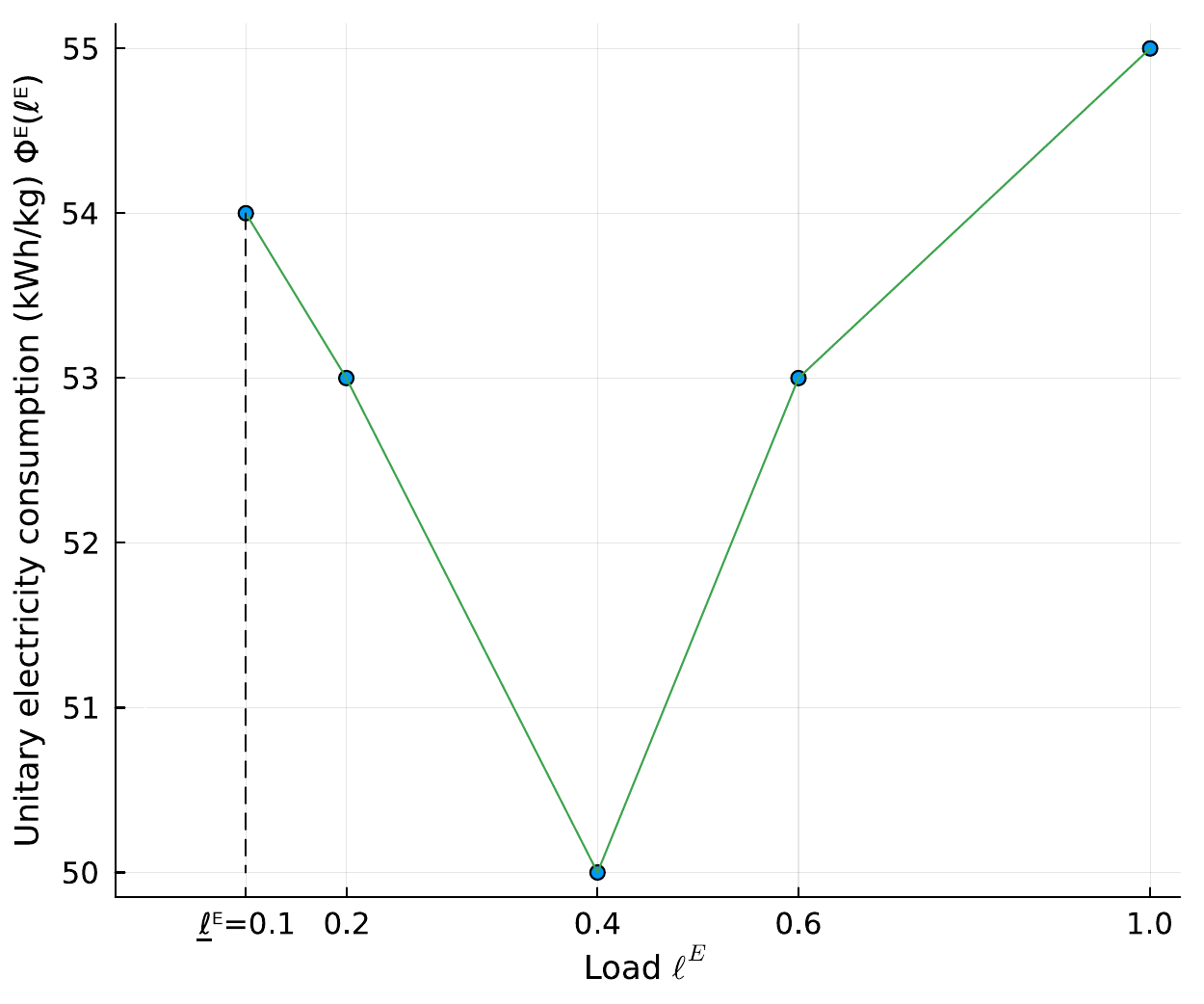}}%
      \caption{\label{fig:schiever_load}}
    \end{subfigure}\hspace{0.5cm}%
    \begin{subfigure}[t]{0.3\textwidth}
      \centering
      \mbox{
    \begin{tabular}{|c|c|c|c|}
      \hline 
      $\Mode | \TurnElectrolyser$ & \COLD & \IDLE & \START \\ \hline 
      \rule{0pt}{2.5ex}  
      \COLD& 1& $\frac{5}{6}$ & $\frac{99}{120}$  \\ [0.6ex] \hline
      \rule{0pt}{2.5ex}    
      \IDLE& $\frac{119}{120}$ & 1 & $\frac{299}{300}$ \\[0.6ex] \hline 
      \rule{0pt}{2.5ex}   
      \START& $\frac{119}{120}$ & $\frac{119}{120}$ & 1 \\[0.6ex] \hline 
    \end{tabular}}
    \caption{\label{table:Schiever_transition_time_electrolyser}}
  \end{subfigure}
  \caption{In~\eqref{fig:schiever_load}, we draw the evolution of electricity consumption as a function of the load of the electrolyser
    and in~\eqref{table:Schiever_transition_time_electrolyser} the values of the function $\mu$ given $\Mode$ and $\TurnElectrolyser$}
\end{figure}

\item The compressor consumption  $e^{\Compressor}$ (used in Equation~\eqref{eq:compressor_electricity}) is equal to \numprint[kWh/kg]{6}.
\item The minimal capacity of the storage, $\underline{\Stock}$, is \numprint[kg]{25} and its maximal capacity, $\overline{\Stock}$, is~\numprint[kg]{750} (see Equation~\eqref{eq:stock_bounded}).

\item The electricity sources have the following characteristics.
  \begin{itemize} 
  \item The stock of PPA, $\overline{\Electricity^{\PPA}}$ (used in Equation~\eqref{eq:ppa}), is equal to \numprint[kWh]{41650}.
  \item The unitary price of PPA (used in Equation~\eqref{eq:Shriver_deterministic_InstantaneousCost1}) is equal to \numprint{0.075}\euro/kWh.
  \item The subsidy $c^s$ is $\numprint{5e6}$\euro, and the ratio of grid electricity $p$ is \numprint{0.2} (used in Equation~\eqref{eq:Schiever_final_cost}).
  \item Figure~\ref{fig:Grid_cost} shows the time evolution of the grid cost.
  \item For a given $\hour \in \HOUR$, the PV energy produced during the time
    interval $[\hour,\hour+1[$, $\va{\Electricity}^{\PV}_{\hour+1}$, is a random
    variable with a discrete probability distribution displayed in
    Table~\ref{tab:PV_discrete_distribution}, where the set of parameters
    $(\mu^{pv}_{\hour+1})_{\hour \in \HOUR}$ are given in
    Figure~\ref{fig:pv_production}.
    \begin{figure}
      \centering
      \begin{subfigure}[b]{0.3\textwidth}
        \centering
        \includegraphics[width=\textwidth]{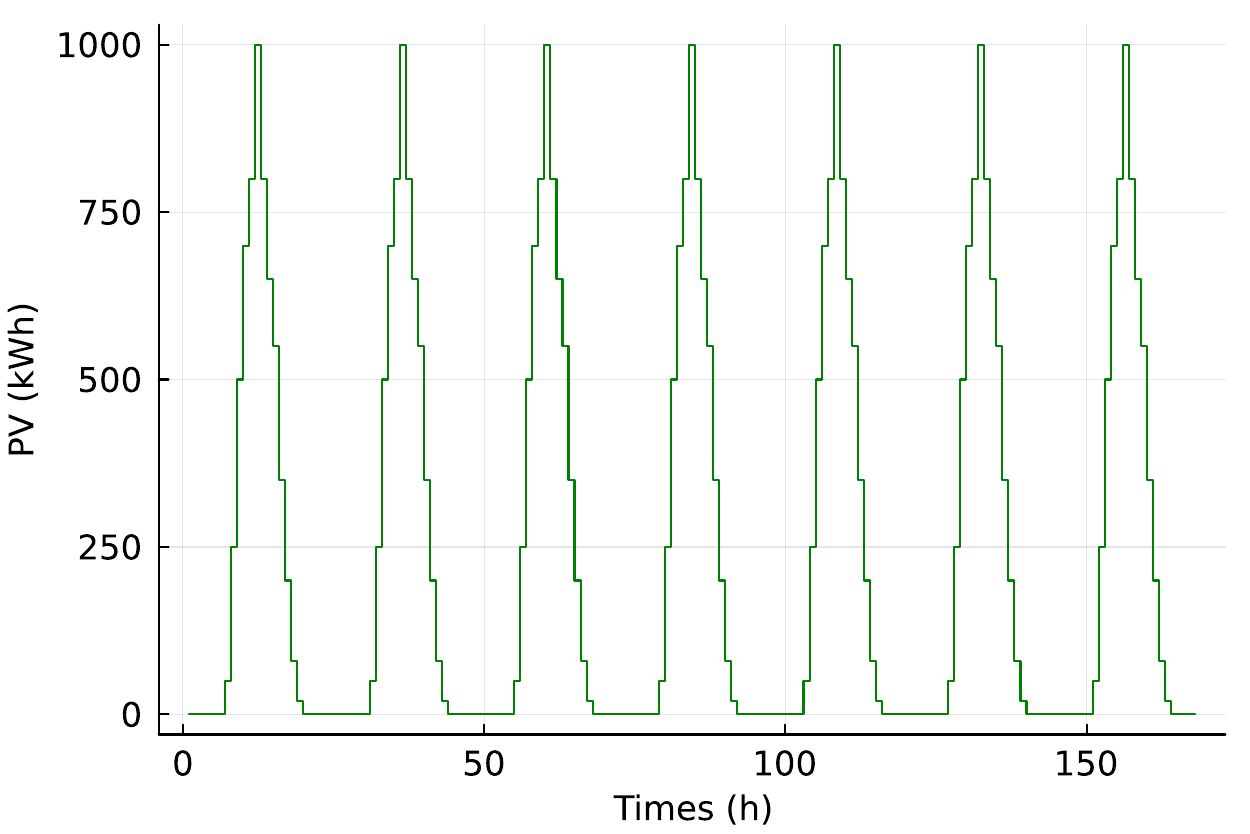}
        \caption{}
        \label{fig:pv_production}
      \end{subfigure}%
      \begin{subfigure}[b]{0.3\textwidth}
        \centering
        \includegraphics[width=\textwidth]{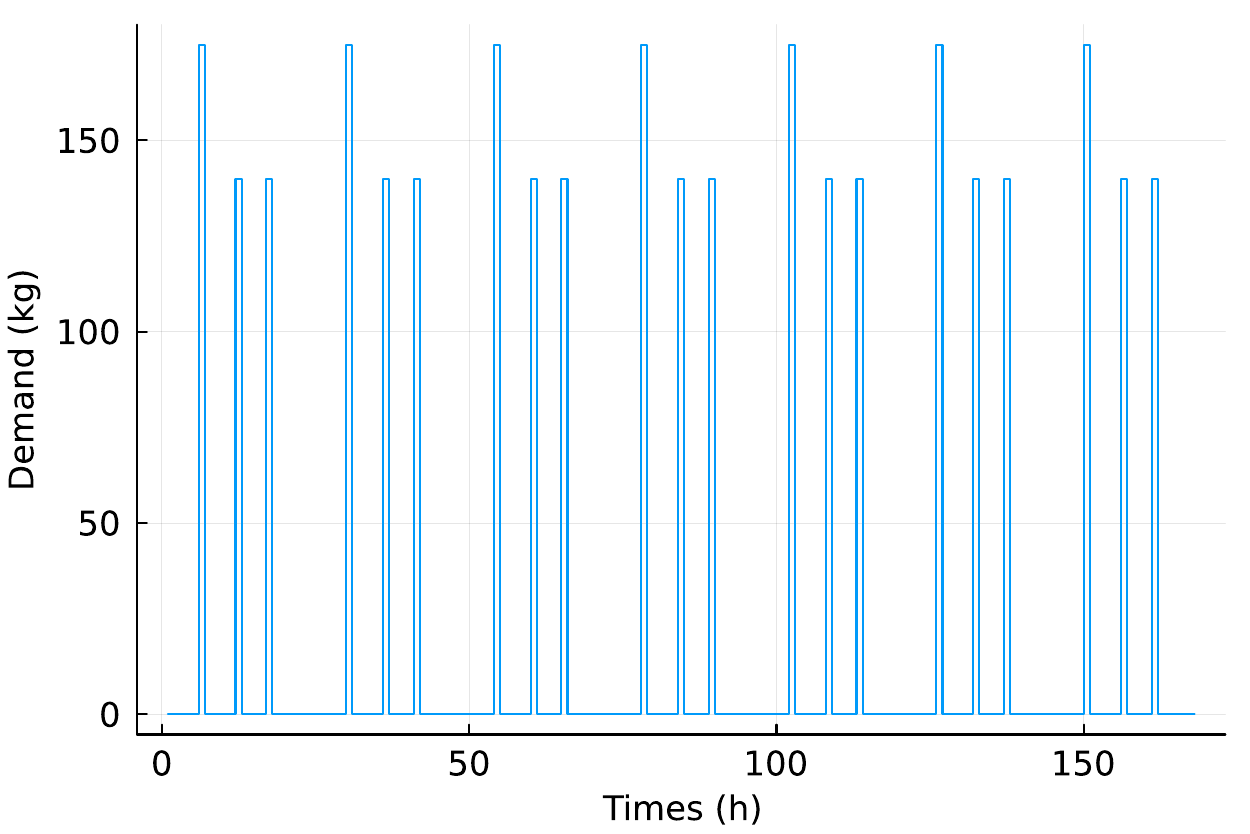}
        \caption{}
        \label{fig:demand}
      \end{subfigure}
      \begin{subfigure}[b]{0.3\textwidth}
        \centering
        \includegraphics[width=\textwidth]{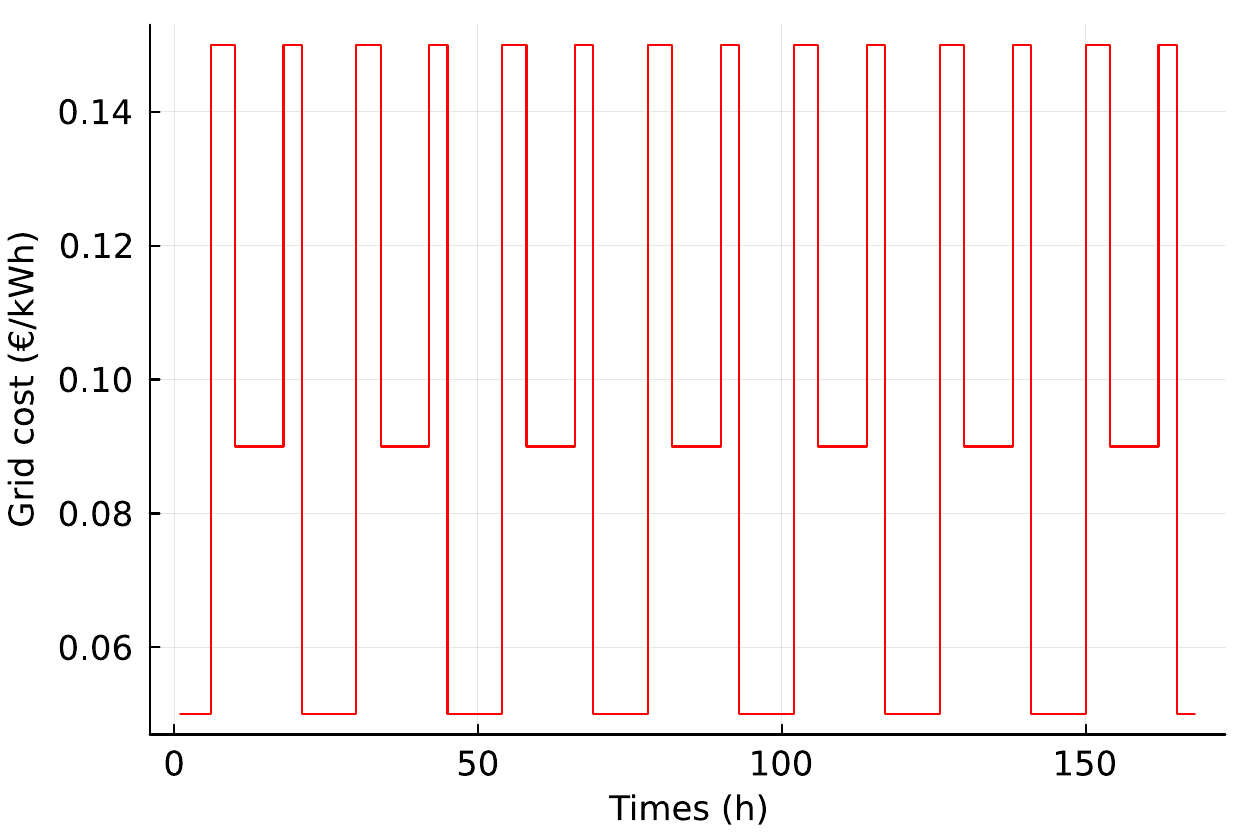}
        \caption{}
        \label{fig:Grid_cost}
      \end{subfigure}
      \caption{Values of the parameters $(\mu^{pv}_{\hour+1})_{\hour \in \HOUR}$ (a), $(\mu^{d}_{\hour+1})_{\hour \in \HOUR}$~(b) and the grid cost $(c^{\Grid}_{\hour+1})_{\hour \in \HOUR}$~(c)}
      \label{fig:grid_pv_evolution}
    \end{figure}


    \begin{table}[htbp]
      \centering
      \caption{Probability distribution of $\va{\Electricity}_{\hour+1}^{\PV}$}
      \begin{tabular}{*{13}{c}}
        \hline
        Outcome & \numprint{0.8}$\mu^{pv}_{\hour+1}$ & \numprint{0.9}$\mu^{pv}_{\hour+1}$ & $\mu^{pv}_{\hour+1}$ & \numprint{1.1}$\mu^{pv}_{\hour+1}$ & \numprint{1.2}$\mu^{pv}_{\hour+1}$ \\
        \hline
        Probability & $\frac{1}{5}$ & $\frac{1}{5}$ & $\frac{1}{5}$ & $\frac{1}{5}$ & $\frac{1}{5}$ \\
        \hline
      \end{tabular}
      \label{tab:PV_discrete_distribution}
    \end{table}
  \end{itemize}
  \item The hydrogen demand has the following characteristics.
    \begin{itemize}
    \item For a given $\hour \in \HOUR$, the hydrogen demand during the time
      interval $[\hour,\hour+1[$, $\va{\Demand}_{\hour+1}$, is a random variable
      with a discrete probability distribution displayed in
      Table~\ref{tab:demand_discrete_distribution}, where the set of parameters
      $(\mu^{d}_{\hour+1})_{\hour \in \HOUR}$ are given in Figure~\ref{fig:demand}.

    \item The dissatisfaction cost $c^d$ is equal to \numprint{5000}\euro/kg (see Equation~\eqref{eq:Shriver_deterministic_InstantaneousCost1}).
      \begin{table}[htbp]
        \centering
        \caption{Probability distribution of $\va{\Demand}_{\hour+1}$}
        \begin{tabular}{*{13}{c}}
          \hline
          Outcome & \numprint{0.8}$\mu^{d}_{\hour+1}$ & \numprint{0.9}$\mu^{d}_{\hour+1}$ & $\mu^{d}_{\hour+1}$ & \numprint{1.1}$\mu^{d}_{\hour+1}$ & \numprint{1.2}$\mu^{d}_{\hour+1}$ \\
          \hline
          Probability & $\frac{1}{5}$ & $\frac{1}{5}$ & $\frac{1}{5}$ & $\frac{1}{5}$ & $\frac{1}{5}$ \\
          \hline
        \end{tabular}
        \label{tab:demand_discrete_distribution}
      \end{table}
    \end{itemize}
   
\end{itemize}

\subsection{Implementation of Algorithm~\ref{alg:Schiever_stochastic_decompo_algo}}
Algorithm~\ref{alg:Schiever_stochastic_decompo_algo} is implemented in Julia
1.9.2. For the SDP component, we use our own implementation developed in Julia.
As for the SDDP component, we employ \texttt{SDDP.jl} as the SDDP
solver~\cite{SDDP.jl}, with \texttt{SDDP.jl} utilizing JuMP~\cite{JuMP} as the modeler
and Gurobi 11.0~\cite{gurobi} as the LP solver. All computations were performed on a
Linux system equipped with 4-processor Intel Xeon E5{-}2667, 3.30GHz, with 192 GB of RAM.

\subsection{Numerical results of Algorithm~\ref{alg:Schiever_stochastic_decompo_algo}}\label{shriver-result}
The value of the parameter $\beta_1$ is set to \numprint{0}. The value of parameter $\beta_2$ is set to \numprint{26.5} to penalize grid consumption over 20\% of the total electricity consumption. The decisions $\load_\hour$ and $\Hydrogen^{\InDemand}_\hour$ and the state $\Stock_\hour$ are discretized for all $\hour \in \HOUR$ for solving the operational Problem~\eqref{eq:Schiever_stochastic_decompo_operational} using SDP. In Table~\ref{Schiever:tab_stoch_Bounds_and_cardinality_of_the_variables}, we give the bounds and cardinality of the set of discrete values of these variables. The bounds of the variables $\load_\hour$ and $\Stock_\hour$ are derived from the data of the problem, while the upper bound of $\Hydrogen^{\InDemand}_\hour$ is the maximum possible hydrogen demand.
\begin{table}[h!]
  \centering
  \begin{tabular}{||c c c c c||}
    \hline
    Variable &Lower bound & Upper bound & Cardinality& \\ [0.5ex] 
    \hline\hline
    $\load_\hour$ & $\MinLoad$=0.1 &1&30 & \\ 
    \hline
    $\Hydrogen^{\InDemand}_\hour$~(kg) &0 & $1.2\mu^{d}_{\hour+1}$ &7& \\
    \hline
    $\Stock_\hour$~(kg) & 25 & 750 & 300& \\
    \hline 
  \end{tabular}
  \caption{Bounds and cardinality of the set of discrete values of the variables of the operational Problem~\eqref{eq:Schiever_stochastic_decompo_operational}}
  \label{Schiever:tab_stoch_Bounds_and_cardinality_of_the_variables}
\end{table}

The SDDP algorithm is iterated \numprint{60} times, which is enough to obtain a small duality gap as discussed later. The gradient step is initialized at \numprint{5e-6} and diminishes by half each 15 iterations across a span of 51 iterations. The number of Monte-Carlo simulation to compute the gradient-like of the function $\widehat{\phi}^{\EL}[\widehat{\FinalCost}]$ and the total cost $\mathrm{val}(\mathcal{P_{\pi^{\lambda}}}[\widehat{\FinalCost}])$ using the policy~\eqref{eq:admissible-policy-stoch} are \numprint{2300} and \numprint{5000} respectively.

We display in Figure~\ref{fig:Schiever_dual_function_stochastic} the value of the function $\widehat{\phi}[\widehat{\FinalCost}]$ at each iteration of Algorithm~\ref{alg:Schiever_stochastic_decompo_algo} and the value returned by the policy~\eqref{eq:admissible-policy-stoch} when applied to primal problems $\mathrm{val}(\mathcal{P_{\pi^{\lambda}}}[\widehat{\FinalCost}])$ and $\mathrm{val}(\mathcal{P_{\pi^{\lambda}}}[\FinalCost])$ each 5 iterations. 

As $\widehat{\FinalCost}$ highly penalizes positive values of $\ElectricityCumul_\horizon$, it is expected to obtain nonpositive values for $\ElectricityCumul_\horizon$ when using the policy $\pi^\lambda$ and therefore to obtain the subsidy $c^s$. However, during some simulations of the policy $\pi^\lambda$ before iteration 30, the subsidy $c^s$ is not obtained, leading to the high values observed in the red curve. After iteration 30, $\ElectricityCumul_\horizon$ takes nonpositive values, and, as the same policy $\pi^\lambda$ is applied to both problems $\mathcal{P}[\widehat{\FinalCost}]$ and $\mathcal{P}[\FinalCost]$, the red and green curves coincide.
\begin{figure}[htpp]
  \centering
  \mbox{\includegraphics[width=0.6\textwidth]{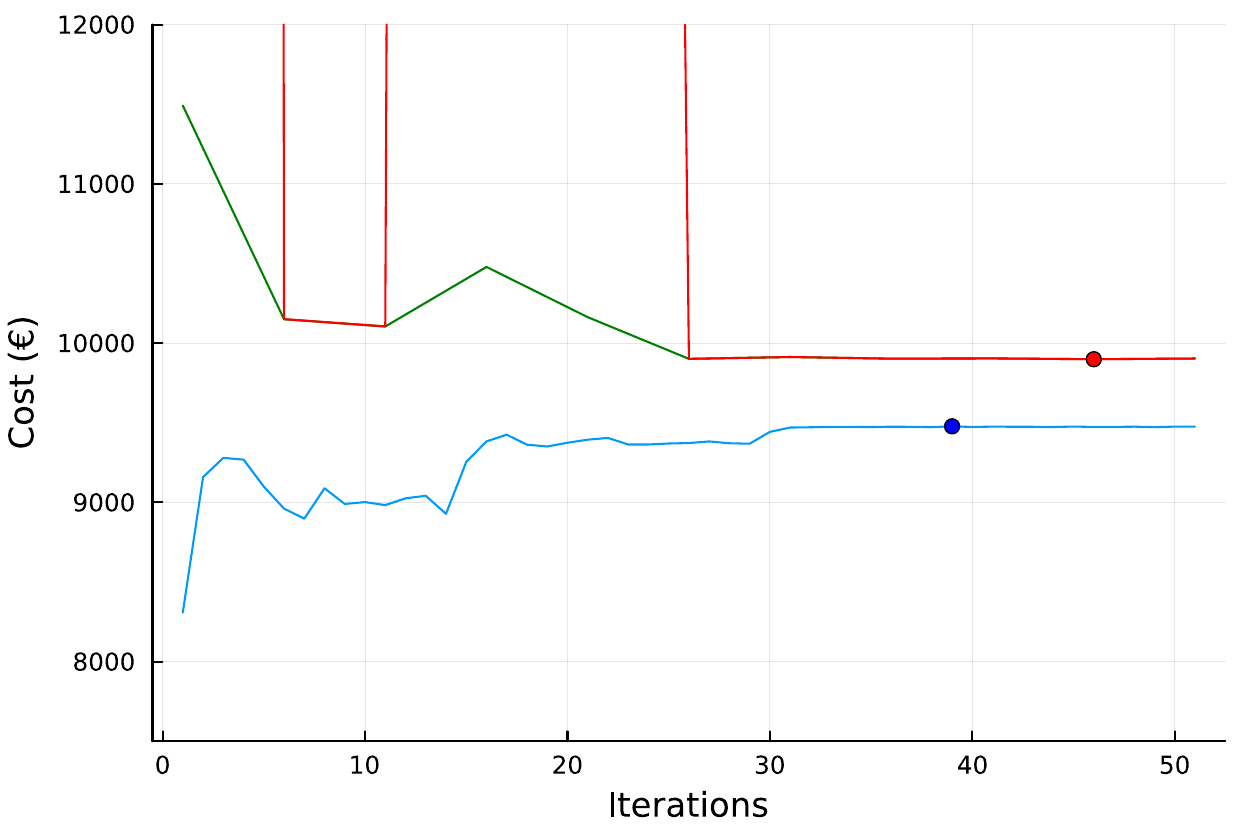}}
  \caption{The blue curve represents the evolution of the function $\widehat{\phi}[\widehat{\FinalCost}]$  along the iterations when using Algorithm~\ref{alg:Schiever_stochastic_decompo_algo}. The green and red curves represent the evolution of the primal problems $\mathcal{P_{\pi^{\lambda}}}\big[\widehat{\FinalCost}]$ and $\mathcal{P_{\pi^{\lambda}}}\big[\FinalCost]$ respectively when applying the policy $\pi^{\lambda}$. The blue point represents the maximal value obtained for the function $\widehat{\phi}[\widehat{\FinalCost}]+c^s$. The red point represents the minimal value obtained for $\mathrm{val}(\mathcal{P_{\pi^{\lambda}}}[\FinalCost])+c^s$. The costs displayed on the vertical axis are the real costs added with the subsidy $c^s$ to ease the reading}
  \label{fig:Schiever_dual_function_stochastic}
\end{figure}

As shown in Equation~\eqref{eq:comparison_dual_primal_stochastic}, the true duality gap $\mathrm{val}(\mathcal{P}[\FinalCost])-\mathrm{val}(\mathcal{D}[\FinalCost])$ is bounded by the difference 
$\mathrm{val}(\mathcal{P_{\pi^{\lambda}}}[\FinalCost])- \mathrm{val}(\mathcal{D}[\widehat{\FinalCost}]) $, that is, the difference between the red and blue points in Figure~\ref{fig:Schiever_dual_function_stochastic}.

This difference of \numprint{418}\euro~is \numprint[\%]{4} of the minimal value obtained for $\mathrm{val}(\mathcal{P_{\pi^{\lambda}}}[\FinalCost])$. This implies that Algorithm~\ref{alg:Schiever_stochastic_decompo_algo} gives good results and that the policy~\eqref{eq:admissible-policy-stoch} for Problem~\eqref{eq:Schiever_stochastic_formulation} is \numprint[\%]{4} optimal.

\subsection{Analysis of some scenarios}
We use the policy~\eqref{eq:admissible-policy-stoch} to simulate the evolution of the stock of hydrogen for an initial stock of \numprint[kg]{250} and for three different scenarios. The evolution of the hydrogen stock is displayed in Figure~\ref{fig:Schiever_stochastic_stock}. We note that the hydrogen stock for scenario 2 reaches high values during the first half of the week. This observation implies a low demand for hydrogen within this scenario. Additionally, it suggests that a maximal capacity of at least \numprint[kg]{500} is required to execute the policy effectively. It is noteworthy that nearly all the hydrogen in the storage is depleted by the end of the horizon for each scenario, which is expected, given the absence of final cost associated with hydrogen stock.

\begin{figure}[htpp]
  \centering
  \mbox{\includegraphics[width=0.50\textwidth]{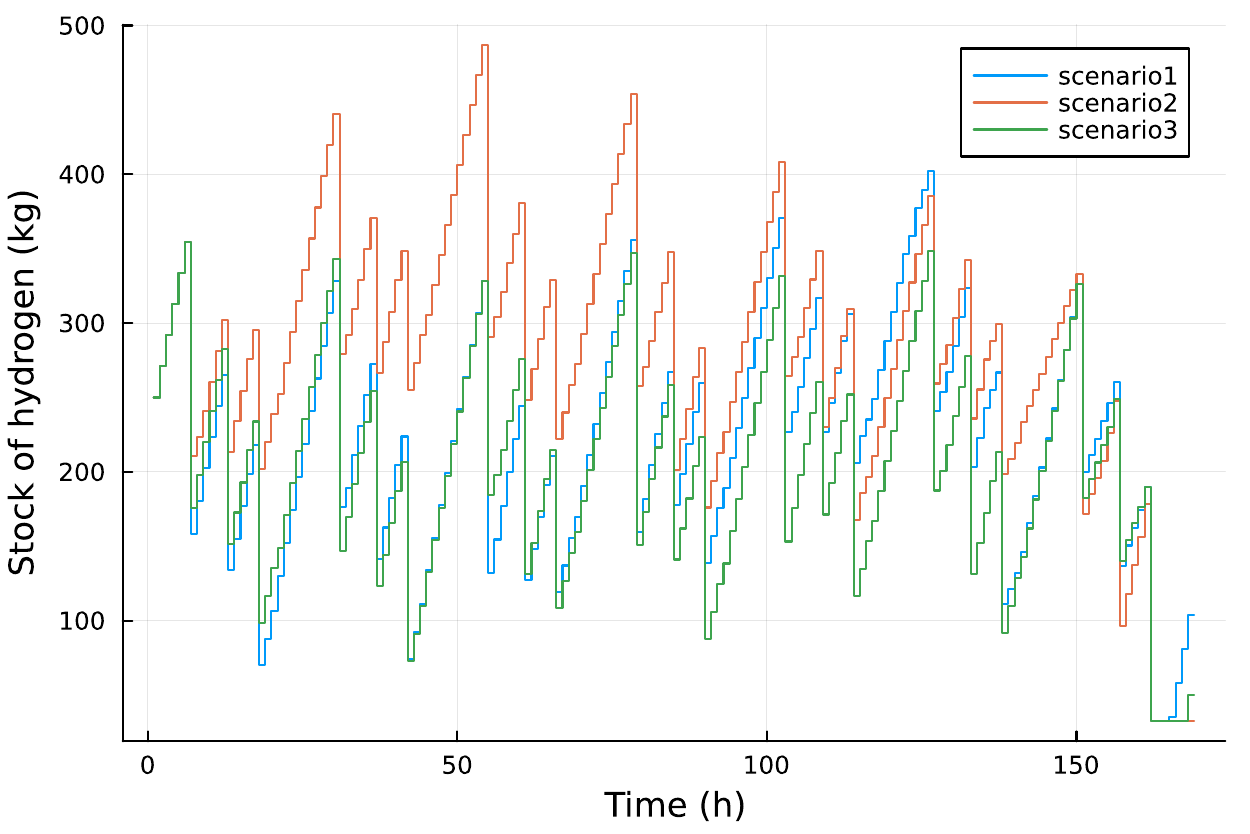}}
  \caption{Time evolution of the optimal stock for the three different scenarios}
  \label{fig:Schiever_stochastic_stock}
\end{figure}

We show in Figure~\ref{fig:Schiever_stochastic_demand_satisfaction} the evolution of the demand and the quantity of hydrogen extracted from the storage to satisfy the hydrogen demand for the three scenarios. As expected, there is low demand during the first half of the week for scenario 2, which explains the large stock of hydrogen during the same period. Note that the demand is always satisfied for every scenario.

\begin{figure}
  \centering
  \begin{subfigure}[b]{0.32\textwidth}
    \centering
    \includegraphics[width=\textwidth]{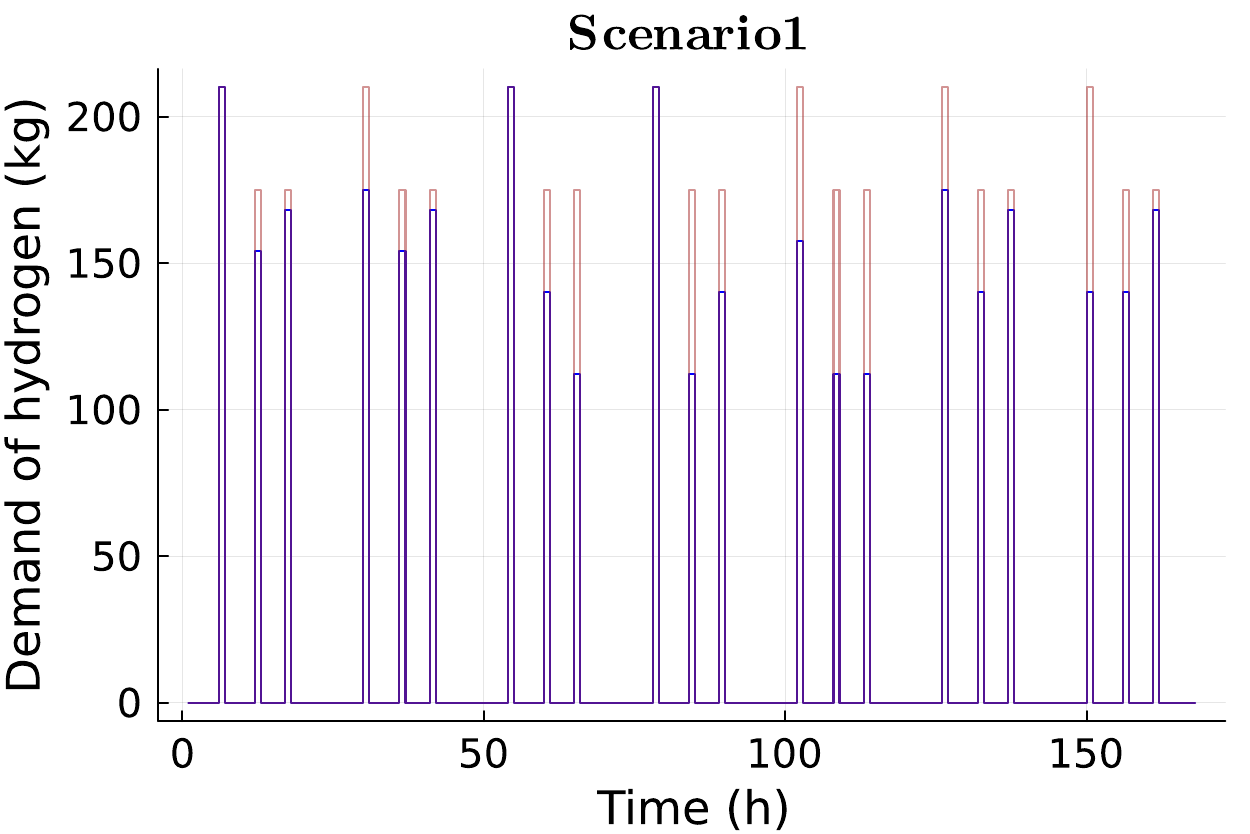}
    \captionsetup{labelformat=empty}
  \end{subfigure}%
  \hspace{0.15cm}\begin{subfigure}[b]{0.32\textwidth}
    \centering
    \includegraphics[width=\textwidth]{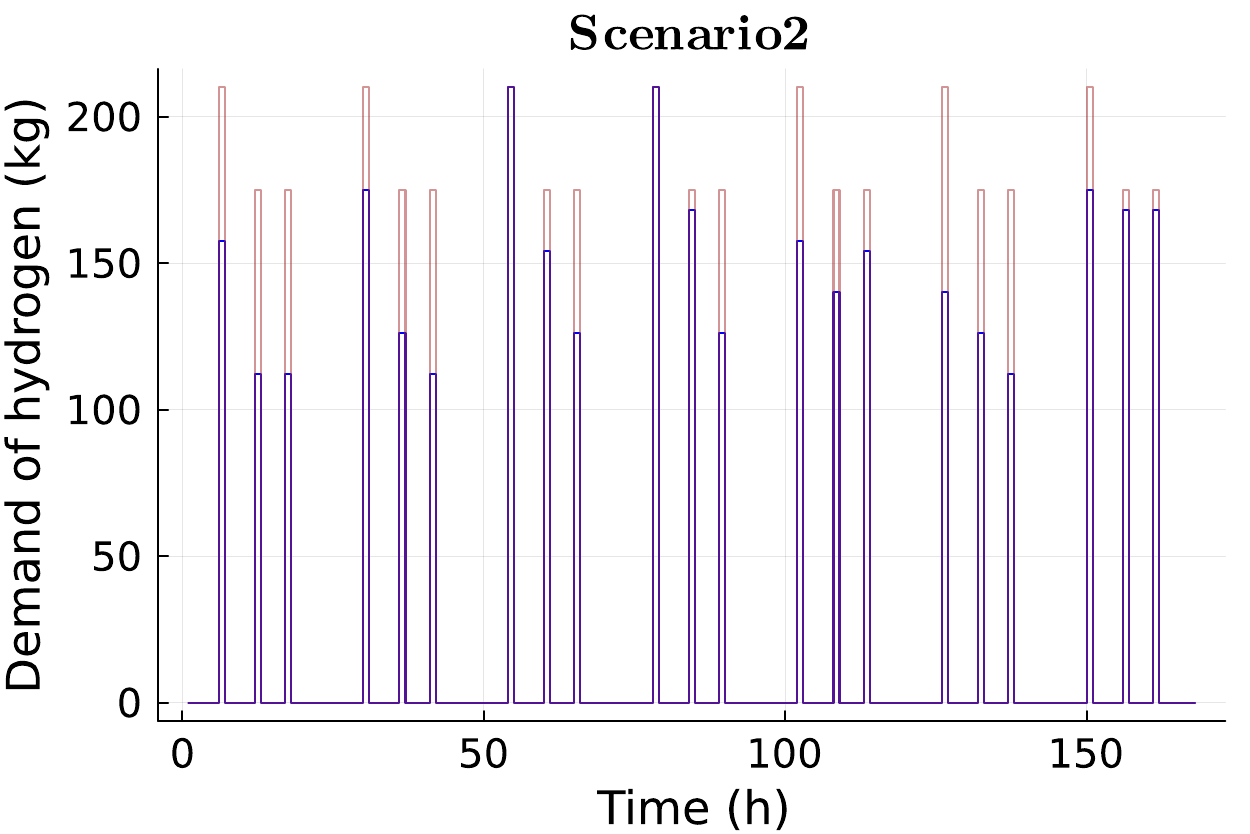}
    \captionsetup{labelformat=empty}
  \end{subfigure}
  \hspace{0.15cm}\begin{subfigure}[b]{0.32\textwidth}
    \centering
    \includegraphics[width=\textwidth]{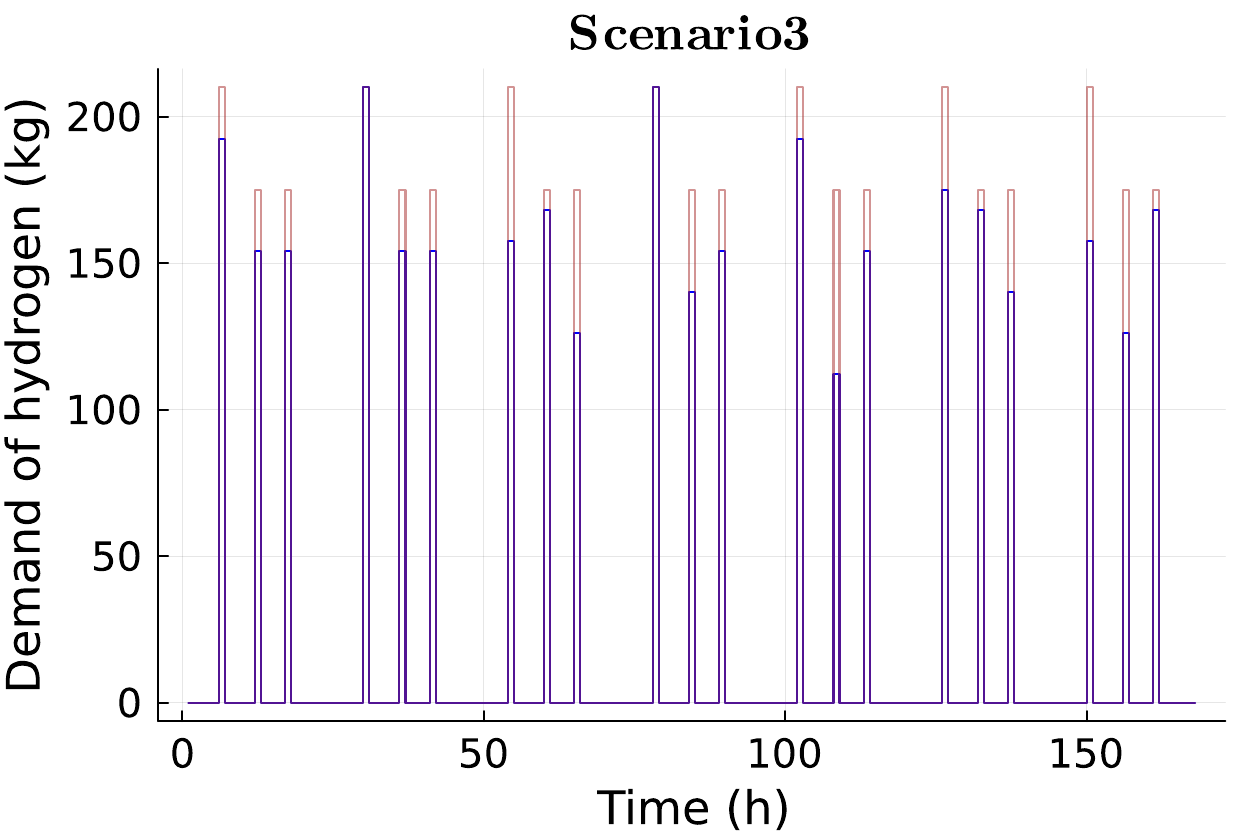}
    \captionsetup{labelformat=empty}
  \end{subfigure}
  \caption{Demand satisfaction for the three different scenarios. The brown bars correspond to the quantity of hydrogen $\Hydrogen^{\InDemand}$ extracted from the storage to satisfy the demand}
  \label{fig:Schiever_stochastic_demand_satisfaction}
\end{figure}

In Figure~\ref{fig:Schiever_stochastic_electricity_consumption}, the electricity consumption is displayed for the three scenarios. We note a pattern where grid electricity is predominantly utilized during the night, taking advantage of its lowest cost. Conversely, PV electricity generation aligns with daytime consumption. For both day and night periods, we complement our energy requirements with PPA. At the end of the week, grid electricity consumption falls below \numprint[\%]{20} of the total electricity consumption, leading to the acquisition of the subsidy $c^s$ for each scenario.

\begin{figure}[hbtp]
  \centering
  \begin{subfigure}[b]{0.32\textwidth}
    \centering
    \includegraphics[width=\textwidth]{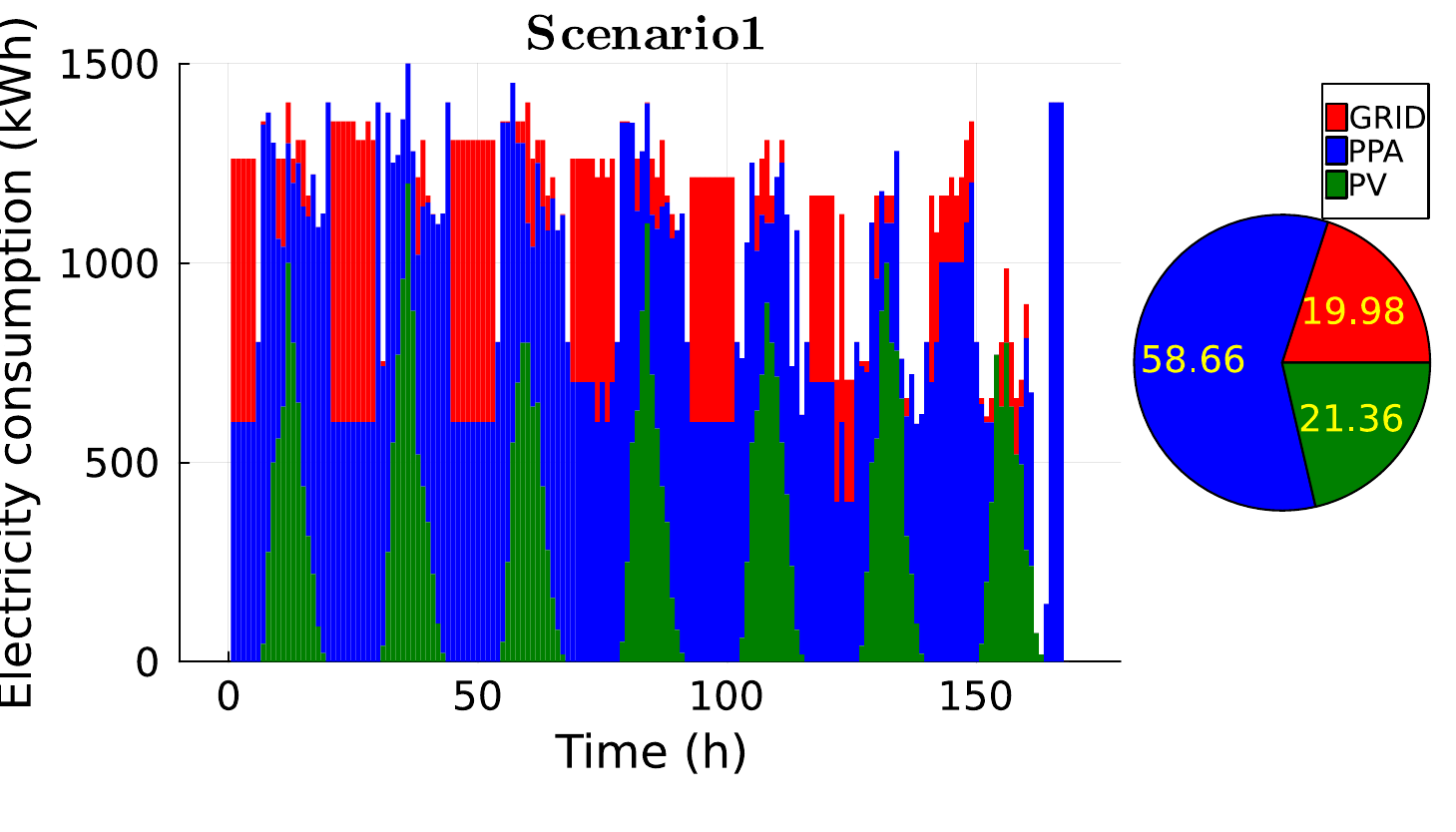}
    \captionsetup{labelformat=empty}
  \end{subfigure}%
  \hspace{0.15cm}\begin{subfigure}[b]{0.32\textwidth}
    \centering
    \includegraphics[width=\textwidth]{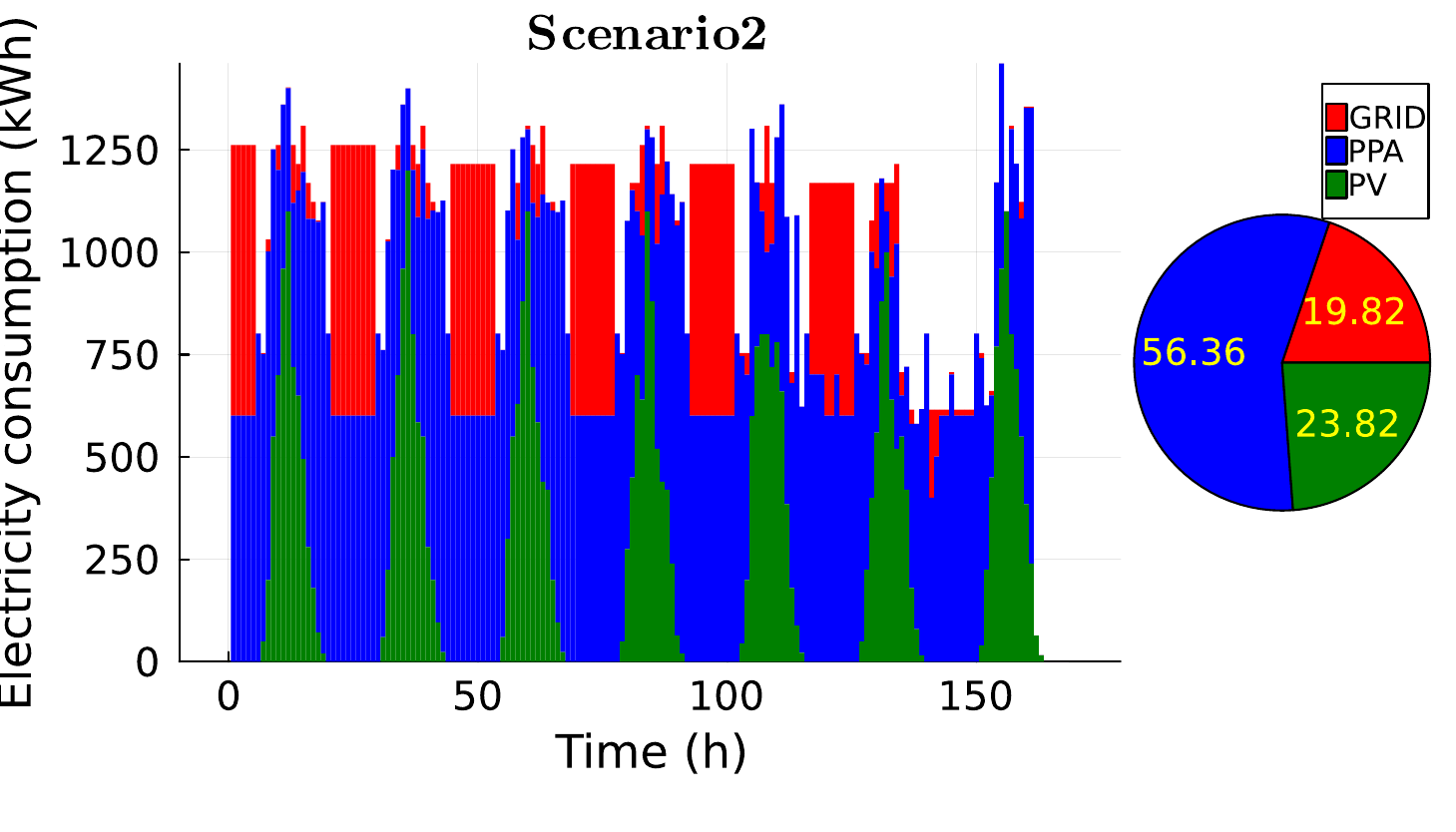}
    \captionsetup{labelformat=empty}
  \end{subfigure}
  \hspace{0.15cm}\begin{subfigure}[b]{0.32\textwidth}
    \centering
    \includegraphics[width=\textwidth]{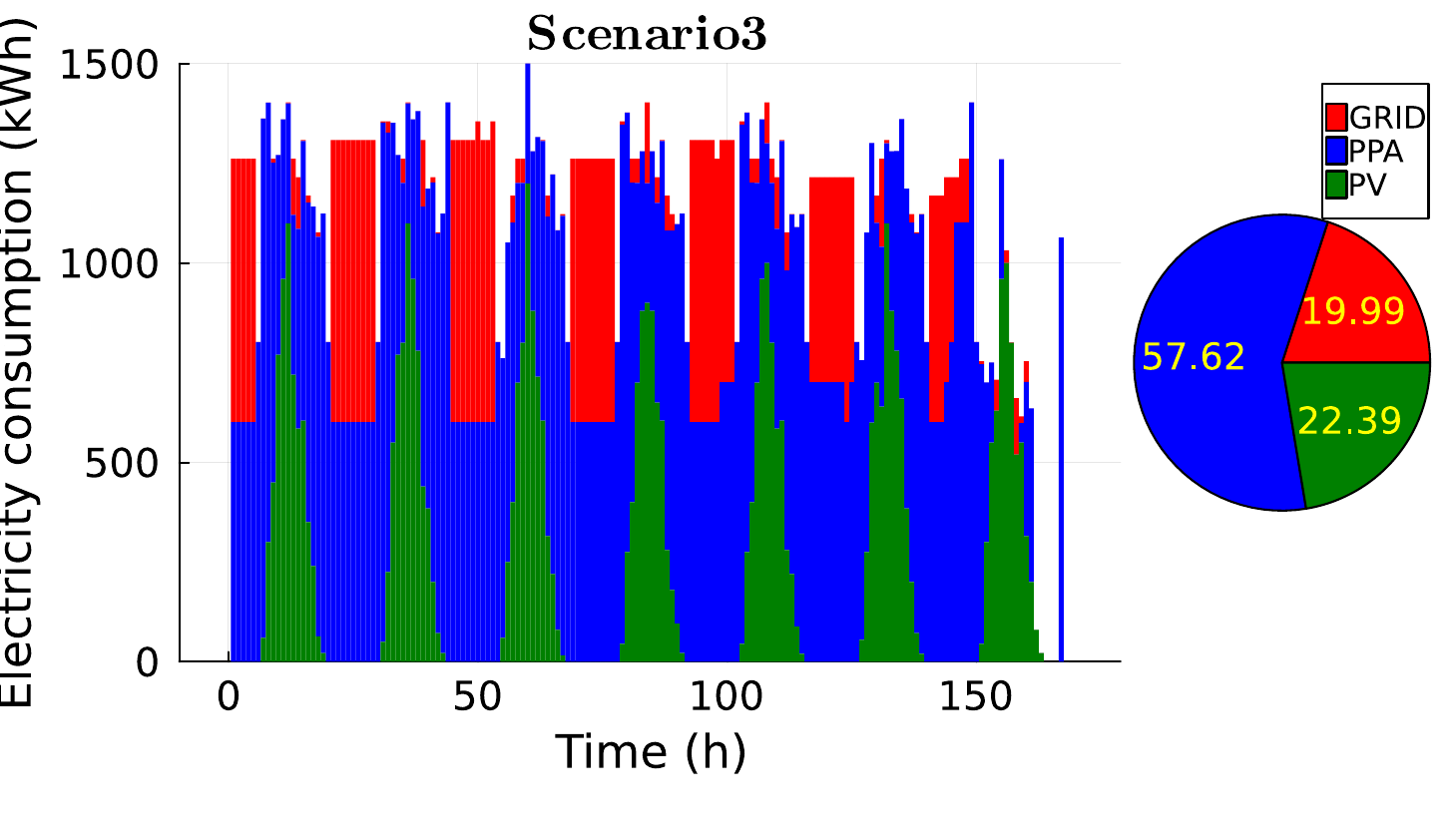}
    \captionsetup{labelformat=empty}
  \end{subfigure}
  \caption{Time evolution of the electricity consumption for the three different scenarios}
  \label{fig:Schiever_stochastic_electricity_consumption}
\end{figure}

\section{Conclusion}\label{Schiever_conclusion}
In this study, we have modeled a hydrogen infrastructure consisting of an electrolyser, compressor, and storage. This infrastructure is powered by a mix of electricity sources, namely PV (Photovoltaic), PPA (Power Purchase Agreement) and the grid, and is managed on an hourly basis to meet hydrogen demand. To address the uncertainties in photovoltaic production and hydrogen demand, we formulated the problem as a multistage stochastic optimization problem. Assuming stagewise independence of noise, we have developed a decomposition algorithm mixed with dynamic programming to solve the problem. Our numerical results are encouraging, with the resultant policy achieving a duality gap of \numprint[\%]{4}. Subsequently, we have analyzed the outcomes of simulations conducted under this policy.

\paragraph{Data Availability} All data used in the numerical experiments is generated as described in this manuscript.

\appendix
\section{Proof of Proposition~\ref{Schiever_proposition}}\label{Schiever_appendix_lemma}

\subsection{Proof of item~\ref{proposition_item1} of Proposition~\ref{Schiever_proposition}}
\label{proof:linearization_of_Q_dynamics}

We start by preliminary notations and results. 
First, we consider a
sequence~$\va{U}=\big(\va{\Electricity}^{\PPA}_{\hour}$, $\va{\Electricity}^{\Grid}_{\hour+1}$, $\va{\PPAStock}_{\hour+1}$, $
\va{\ElectricityCumul}_{\hour+1} \big)_{\hour \in \HOUR}$ admissible for the
Problem~$[\overline{\FinalCost}]${-}\eqref{eq:Schiever_stochastic_decompo_electricity}.
It is straightforward to check that the derived sequence denoted by $\gamma\np{\va{U}}$ and defined by

\begin{equation}
  \gamma\np{\va{U}}
  := \big(\va{\Electricity}^{\PPA}_{\hour},\va{\Electricity}^{\Grid}_{\hour+1},\va{\PPAStock}_{\hour+1},
  \va{\ElectricityCumul}_{\hour+1},(\va{\Electricity}^{\Grid}_{\hour+1})_+,
  \min(\overline{\Electricity},\va{\Electricity}^{\PPA}_{\hour}+\va{\Electricity}^{\PV}_{\hour+1}) \big)_{\hour \in \HOUR}
  \label{eq:derivedone}
\end{equation}
is admissible for Problem~$[\overline{\FinalCost}]${-}\eqref{eq:Schiever_stochastic_decompo_electricity2}.
Moreover, as the respective costs of Problem~$[\overline{\FinalCost}]$-~\eqref{eq:Schiever_stochastic_decompo_electricity} and
Problem~$[\overline{\FinalCost}]$-\eqref{eq:Schiever_stochastic_decompo_electricity2} have the same expression
only depending on the sequence
$\big(\va{\Electricity}^{\PPA}_{\hour},\va{\Electricity}^{\Grid}_{\hour+1},\va{\ElectricityCumul}_{\hour+1} \big)_{\hour \in \HOUR}$,
we obtain that the two sequence gives the same cost
\begin{equation}
  \mathbb{E} \Big[ \sum\limits_{\hour \in \HOUR}  \InstantaneousCost_\hour^{\EL}\bp{\va{\Electricity}^{\PPA}_\hour, \va{\Electricity}^{\Grid}_{\hour+1},\va{\Electricity}^{\PV}_{\hour+1},\lambda_\hour} +\widehat{\FinalCost}(\va{\ElectricityCumul}_\horizon) \Big]
  \eqfinp
\end{equation}

Second, we consider a sequence
$\va{V}=\big(\va{\Electricity}^{\PPA}_{\hour},\va{\Electricity}^{\Grid}_{\hour+1},\va{\PPAStock}_{\hour+1},
\va{\ElectricityCumul}_{\hour+1},\va{\Electricity}^{N}_{\hour+1}, \va{\Electricity}^{R}_{\hour+1} \big)_{\hour \in \HOUR}$
admissible for Problem~$[\overline{\FinalCost}]$-\eqref{eq:Schiever_stochastic_decompo_electricity2}.
We build the sequence $(\overline{\va{\ElectricityCumul}}_{\hour})_{\hour \in \HOUR}$ defined, for all
${\hour \in \HOUR}$ by 
\begin{equation}
  \overline{\va{\ElectricityCumul}}_0=0 \eqsepv
  \overline{\va{\ElectricityCumul}}_{\hour+1} =
  \overline{\va{\ElectricityCumul}}_\hour +(1-p)(\va{\Electricity}_{\hour+1}^{\Grid})_+
  -p\min(\overline{\Electricity},\va{\Electricity}^{\PPA}_\hour+\va{\Electricity}_{\hour+1}^{\PV}) \eqfinv
  \label{eq:Qbardef}
\end{equation}
and denote by $\varphi\np{\va{V}}$ the sequence
\begin{equation}
  \varphi\np{\va{V}}:= \big(\va{\Electricity}^{\PPA}_{\hour},\va{\Electricity}^{\Grid}_{\hour+1},\va{\PPAStock}_{\hour+1},
  \overline{\va{\ElectricityCumul}}_{\hour+1}) \big)_{\hour \in \HOUR}
  \label{eq:theta}
  \eqfinp
\end{equation}
We straightforwardly check that the derived sequence $\gamma(\varphi(\va{V}))$ is
admissible for Problem~$[\overline{\FinalCost}]$-\eqref{eq:Schiever_stochastic_decompo_electricity2} and
gives a lower cost. Indeed, it follows from~Equation~\eqref{eq:Qbardef} and~\eqref{eq:schiever_stochastic_linearization_constraints}
that $\overline{\va{\ElectricityCumul}}_{\horizon} \le {\va{\ElectricityCumul}}_{\horizon}$ which combined with the fact that
the function $\overline{\FinalCost}$ is nondecreasing gives a lower cost for the derived sequence.
Moreover, using the first part, we also have that $\varphi(\va{V})$ is admissible for
Problem~$[\overline{\FinalCost}]${-}\eqref{eq:Schiever_stochastic_decompo_electricity} with the same cost as the one
given by $\gamma(\varphi(\va{V}))$ in Problem~$[\overline{\FinalCost}]$-\eqref{eq:Schiever_stochastic_decompo_electricity2}.

\medskip
Now, we turn to the proof of item~\ref{proposition_item1} of Proposition~\ref{Schiever_proposition}.
We denote by $h^{\EL}$ (resp $\widehat{h}^{\EL}$) the cost function of Problem~$[\overline{\FinalCost}]${-}\eqref{eq:Schiever_stochastic_decompo_electricity} (resp. Problem~$[\overline{\FinalCost}]$-\eqref{eq:Schiever_stochastic_decompo_electricity2}).

We consider $\lambda \in \mathbb{R}^\horizon$ and assume that $\phi^{\EL}[\widehat{\FinalCost}](\lambda)$,
the optimal cost of Problem~$[\overline{\FinalCost}]$-\eqref{eq:Schiever_stochastic_decompo_electricity}, is finite.
For $\xi>0$, consider $\va{U}^{\xi}$ a $\xi$-optimal solution of Problem~$[\overline{\FinalCost}]${-}\eqref{eq:Schiever_stochastic_decompo_electricity}. Using the preliminary part, we have that
\begin{align*}
  \widehat{\phi}^{\EL}[\widehat{\FinalCost}](\lambda)
  \le \widehat{h}^E \bp{\gamma(\va{U}^{\xi})} = h^E \bp{\va{U}^{\xi})} \le \phi^{\EL}[\widehat{\FinalCost}](\lambda) + \xi
  \eqfinv
\end{align*}
which gives that $\widehat{\phi}^{\EL}[\widehat{\FinalCost}](\lambda) \le  \phi^{\EL}[\widehat{\FinalCost}](\lambda)$ and therefore
$\widehat{\phi}^{\EL}[\widehat{\FinalCost}](\lambda) < +\infty$.
Now, first assume that $\widehat{\phi}^{\EL}[\widehat{\FinalCost}](\lambda)$ is finite and consider $\va{V}^{\xi}$ a
$\xi$-optimal solution for Problem~$[\overline{\FinalCost}]$-\eqref{eq:Schiever_stochastic_decompo_electricity2}.
We obtain that
\begin{align*}
  \phi^{\EL}[\widehat{\FinalCost}](\lambda)
  \le h^E \bp{\varphi(\va{V}^{\xi})}
  = \widehat{h}^E \bp{\gamma(\varphi(\va{V}^{\xi})}
  \le \widehat{h}^E \bp{\va{V}^{\xi}}  \le \widehat{\phi}^{\EL}[\widehat{\FinalCost}](\lambda) + \xi
  \eqfinv
\end{align*}
which gives the equality $\phi^{\EL}[\widehat{\FinalCost}](\lambda) = \widehat{\phi}^{\EL}[\widehat{\FinalCost}](\lambda)$.
Second, it remains to consider the case $\widehat{\phi}^{\EL}[\widehat{\FinalCost}](\lambda)= -\infty$. For each $n\in \NN$
we can find an admissible $\va{V}_n$ for  Problem~$[\overline{\FinalCost}]${-}\eqref{eq:Schiever_stochastic_decompo_electricity}
satisfying $\widehat{h}^E(\va{V}_n) \le -n$ and proceeding as above we $\phi^{\EL}[\widehat{\FinalCost}](\lambda) \le -n$ for all $n \in \NN$ which contradict the assumption that $\phi^{\EL}[\widehat{\FinalCost}](\lambda)$ is finite.

The cases $\phi^{\EL}[\widehat{\FinalCost}](\lambda) \in \na{+\infty, -\infty}$ can be treated in a similar way and are left to the reader.

\subsection{Proof of item~\ref{proposition_item2} of Proposition~\ref{Schiever_proposition}}\label{proof:ElectricityCumul_bound}
\begin{proof}
Fix $\lambda \in \mathbb{R}^\horizon$ and consider the optimization Problem~\eqref{eq:Schiever_stochastic_decompo_electricity}. As a preliminary fact, we prove that for a feasible solution of Problem~\eqref{eq:Schiever_stochastic_decompo_electricity}, its component ``cumulative electricity'' at time $\horizon$ satisfies $\va{\ElectricityCumul}_\horizon \in [\underline{\ElectricityCumul}, \overline{\ElectricityCumul}]$. First, using the cumulative electricity state dynamics~\eqref{eq:cumul_dynamics}, we obtain the lower bound
\begin{align}
\va{\ElectricityCumul}_\horizon &= \sum_{\hour= 0}^{\horizon-1} \underbrace{(1-p)(\va{\Electricity}_{\hour+1}^{\Grid})_+}_{\geq 0}-p\min(\overline{\Electricity},\va{\Electricity}^{\PPA}_{\hour}+\va{\Electricity}_{\hour+1}^{\PV}) \nonumber
\\ 
&\geq \sum_{\hour=0}^{\horizon-1} -p\min(\overline{\Electricity},\va{\Electricity}^{\PPA}_{\hour}+\va{\Electricity}_{\hour+1}^{\PV}) \geq \sum_{\hour=0}^{\horizon-1} -p\overline{\Electricity} = -\horizon p\overline{\Electricity}\nonumber \eqfinp 
\intertext{Second, we obtain the upper bound}
\va{\ElectricityCumul}_\horizon &= \sum_{\hour= 0}^{\horizon-1} (1-p)(\va{\Electricity}_{\hour+1}^{\Grid})_+ \underbrace{-p\min(\overline{\Electricity},\va{\Electricity}^{\PPA}_{\hour}+\va{\Electricity}_{\hour+1}^{\PV})}_{\leq 0} \nonumber
\\ 
&\leq  \sum_{\hour= 0}^{\horizon-1} (1-p)(\va{\Electricity}_{\hour+1}^{\Grid})_+ \leq \sum_{\hour=0}^{\horizon-1} (1-p)\overline{\Electricity} = \horizon(1-p)\overline{\Electricity} \nonumber \eqfinp \tag{as $\va{\Electricity}^{\Grid}_{\hour+1} \leq \overline{\Electricity}$ by~\eqref{eq:schiever_total_elec_bounded}}
\end{align}

Now, we prove that $\phi[\overline{\FinalCost}](\lambda) \leq
\phi[\FinalCost](\lambda)$. Fix $\lambda \in \mathbb{R}^\horizon$. The optimal cost of
Problem~\eqref{eq:Schiever_stochastic_decompo_electricity}
$\phi^{\EL}[\overline{\FinalCost}](\lambda)$ is in $\RR \cup \na{+\infty}$ as the
feasible set of Problem~\eqref{eq:Schiever_stochastic_decompo_electricity} is
bounded, the objective function of
Problem~\eqref{eq:Schiever_stochastic_decompo_electricity} is proper as
$\overline{\FinalCost}$ is proper and $(\va{\Electricity}^{\PV}_{\hour+1})_{\hour
\in \HOUR}$ has a finite support. Thus, for a given $\zeta>0$, there exists
$\big(\va{\Electricity}^{\PPA,\zeta}_{\hour},\va{\Electricity}^{\Grid,\zeta}_{\hour+1},\va{\PPAStock}^{\zeta}_{\hour},\va{\ElectricityCumul}^{\zeta}_\hour
\big)_{\hour \in \HOUR}$ in the feasible of
Problem~\eqref{eq:Schiever_stochastic_decompo_electricity} satisfying
\begin{equation}\label{eq:Schiever_zeta_elec}
\mathbb{E} \Big[\sum\limits_{\hour \in \HOUR}  \InstantaneousCost_\hour^{\EL}\bp{\va{\Electricity}^{\PPA,\zeta}_\hour,
      \va{\Electricity}^{\Grid,\zeta}_{\hour+1},\va{\Electricity}^{\PV}_{\hour+1},\lambda_\hour} +\FinalCost(\va{\ElectricityCumul}^{\zeta}_\horizon) \Big] \leq \phi^{\EL}[\FinalCost] + \zeta \eqfinp
  \end{equation}
 
  We immediately obtain that the control $\big(\va{\Electricity}^{\PPA,\zeta}_{\hour},\va{\Electricity}^{\Grid,\zeta}_{\hour+1},\va{\PPAStock}^{\zeta}_{\hour},\va{\ElectricityCumul}^{\zeta}_\hour \big)_{\hour \in \HOUR}$  is feasible for Problem~\eqref{eq:Schiever_stochastic_decompo_electricity} where $\FinalCost$ is replaced by $\overline{\FinalCost}$. Therefore we have
  
\begin{align}
\phi^{\EL}[\overline{\FinalCost}](\lambda) &\leq  \mathbb{E} \Big[\sum\limits_{\hour \in \HOUR}  \InstantaneousCost_\hour^{\EL}\bp{\va{\Electricity}^{\PPA,\zeta}_\hour, \va{\Electricity}^{\Grid,\zeta}_{\hour+1},\va{\Electricity}^{\PV}_{\hour+1},\lambda_\hour} + \overline{\FinalCost}(\va{\ElectricityCumul}^{\zeta}_\horizon) \Big] \nonumber \tag{solution feasibility}
 \\ &\leq \mathbb{E} \Big[ \sum\limits_{\hour \in \HOUR}  \InstantaneousCost_\hour^{\EL}\bp{\va{\Electricity}^{\PPA,\zeta}_\hour,
     \va{\Electricity}^{\Grid,\zeta}_{\hour+1},\va{\Electricity}^{\PV}_{\hour+1},\lambda_\hour} + \FinalCost(\va{\ElectricityCumul}^{\zeta}_\horizon)\Big] \nonumber \tag{$\overline{\FinalCost}\leq \FinalCost$ in $[\underline{\ElectricityCumul},\overline{\ElectricityCumul}]$}
      \\ &\leq \phi^{\EL}[\FinalCost](\lambda) + \zeta \eqfinp \nonumber \tag{using~\eqref{eq:Schiever_zeta_elec}}
\end{align}
We conclude that $\phi^{\EL}[\overline{\FinalCost}](\lambda) \leq \phi^{\EL}[\FinalCost](\lambda) + \zeta$ for all $\zeta >0$, and therefore that\\ $\phi^{\EL}[\overline{\FinalCost}](\lambda) \leq \phi^{\EL}[\FinalCost](\lambda)$. Finally, we have 

\begin{align}
\phi[\overline{\FinalCost}](\lambda) &= \phi^{\EL}\big[\overline{\FinalCost}](\lambda) + \phi^{\OP}(\lambda) \nonumber
\\ 
&\leq \phi^{\EL}[\FinalCost](\lambda) + \phi^{\OP}(\lambda) \nonumber
\\ 
&= \phi[\FinalCost](\lambda) \eqsepv \forall \lambda \in \mathbb{R}^\horizon \eqfinp \nonumber
\end{align}

This ends the proof.
\end{proof}

\subsection{Proof of item~\ref{proposition_item3} of Proposition~\ref{Schiever_proposition}}\label{proof:choice_of_final_cost}
\begin{proof}
  First, we prove that Problem~\eqref{eq:Schiever_stochastic_decompo_electricity2}  is a convex optimization problem. The function $\widehat{\FinalCost}$ is the maximum of affine functions with nonnegative slopes, we conclude that $\widehat{\FinalCost}$ is proper, nondecreasing and convex, which satisfies the assumptions of item~\ref{proposition_item1} of Proposition~\ref{Schiever_proposition}. Consequently, Problem~\eqref{eq:Schiever_stochastic_decompo_electricity2} is a convex optimization problem for all $\lambda \in \mathbb{R}^\horizon$.
  
  Second, we prove that $\widehat{\FinalCost} \leq \FinalCost$ in the interval $[\underline{\ElectricityCumul},\overline{\ElectricityCumul}]$.
 We distinguish the two following cases
\begin{itemize}
\item For $\ElectricityCumul_\horizon \leq 0$ we have

\begin{equation}
\widehat{\FinalCost}(\ElectricityCumul_\horizon) = \beta_1 \ElectricityCumul_\horizon - c^s \leq -c^s = \FinalCost(\ElectricityCumul_\horizon) \eqfinp
\end{equation}
\item Conversely, for $\ElectricityCumul_\horizon > 0$, we have 

\begin{subequations}
\begin{equation}
\widehat{\FinalCost}(\ElectricityCumul_\horizon) = \beta_2 \ElectricityCumul_\horizon - c^s \eqfinv
\end{equation}
\begin{equation}
\FinalCost(\ElectricityCumul_\horizon)=0 \eqfinv
\end{equation}
\begin{equation}
\widehat{\FinalCost}(\ElectricityCumul_\horizon) \leq \FinalCost(\ElectricityCumul_\horizon) \eqsepv \forall \ElectricityCumul_\horizon \in ]0,\overline{\ElectricityCumul}] \iff \beta_2 \leq \frac{c^s}{\ElectricityCumul_\horizon} \eqsepv \forall \ElectricityCumul_\horizon \in ]0,\overline{\ElectricityCumul}] \iff \beta_2 \leq \frac{c^s}{\overline{\ElectricityCumul}} \eqfinp
\end{equation}
\end{subequations}
\end{itemize}
We conclude that if $\beta_2 \leq \frac{c^s}{\overline{\ElectricityCumul}}$ then $\widehat{\FinalCost}\leq \FinalCost$ in $[\underline{\ElectricityCumul},\overline{\ElectricityCumul}]$ and therefore by using item~\ref{proposition_item2} of Proposition~\ref{Schiever_proposition} that $\widehat{\phi}[\widehat{\FinalCost}] \leq \phi[\FinalCost]$. This ends the proof. 
\end{proof}

\section{Initialization of Lagrange multiplier}
\begin{lemma}\label{lemma:schiever_lemma_multiplier}
  Let $f$, $(g_\hour)_{\hour \in \HOUR}$ and $\FinalCost$ be convex functions taking finite values, $C_u$ and $C_y$ be closed convex sets with $C_u \subset \mathbb{R}^{m \times \horizon}$ and $C_y \subset \mathbb{R}^\horizon$,
 and $(\alpha_\hour,\beta_\hour,a_\hour,b_\hour)_{\hour \in \HOUR}$ a sequence of positive parameters. Given $\hour' \in \HOUR$, consider the following optimization Problem for all $\epsilon \geq 0$:
\begin{subequations}\label{eq:schiever_lemma_multiplier}
  \begin{equation}
    \psi_{\hour'}(\epsilon) = \min\limits_{\np{ 
        x_\hour,y_\hour,u_\hour
      }_{\hour \in \HOUR}} \sum\limits_{\hour \in \HOUR}
    \big(\alpha_\hour x_\hour +\beta_\hour y_\hour \big)+ f\big((u_\hour)_{\hour \in \HOUR}\big)+ \FinalCost \big(\sum_{\hour \in \HOUR} a_\hour x_\hour-b_\hour y_\hour\big)
  \end{equation}
  subject to the following constraints
  \begin{align}
  (u_\hour)_{\hour \in \HOUR} &\in C_u \eqfinv
  \\ 
  (y_\hour)_{\hour \in \HOUR} &\in C_y \eqfinv
  \\
  g_{\hour}(u_\hour)&\leq x_\hour + y_\hour  \eqsepv \forall \hour \in \HOUR \setminus \{\hour'\} \eqfinv
  \\ 
   g_{\hour'}(u_{\hour'}) &\leq x_{\hour'} + y_{\hour'}+\epsilon  \eqfinv \label{eq:schiever_lemma_constraint_lambda}
      \\
     0 &\leq x_\hour, y_\hour \eqsepv \forall \hour \in \HOUR \eqfinp
\end{align}
\end{subequations}
We assume that if $(y_\hour)_{\hour \in \HOUR} \in C_y$ then $\{(\hat{y}_{\hour})_{\hour \in \HOUR} |\: \hat{y}_{\hour} \leq y_{\hour} \eqsepv \forall \hour \in \HOUR\} \subset C_y$.
We also assume that for $\epsilon \geq 0$, the Lagrangian of Problem~\eqref{eq:schiever_lemma_multiplier} when dualizing Constraint~\eqref{eq:schiever_lemma_constraint_lambda} admits a saddle point $\big( (x^{\epsilon},y^{\epsilon},u^{\epsilon}), \lambda^\epsilon_{\hour'} \big)$ where  $ (x^{\epsilon},y^{\epsilon},u^{\epsilon}) = (x^{\epsilon}_\hour,y^{\epsilon}_\hour,u^{\epsilon}_\hour)_{\hour \in \HOUR}$ is the optimal solution of Problem~\eqref{eq:schiever_lemma_multiplier}, and $\lambda^{\epsilon}_{\hour'}$ is the Lagrange multiplier associated with constraint~\eqref{eq:schiever_lemma_constraint_lambda}.

If $x^{0}_{\hour'}>0$ and $y^{0}_{\hour'}>0$ then $\lambda^{0}_{\hour'} \geq \frac{ \alpha_{\hour'}b_{\hour'}}{a_{\hour'} + b_{\hour'}} + \frac{\beta_{\hour'}a_{\hour'}}{a_{\hour'} + b_{\hour'}}$.

\end{lemma}

\begin{proof}

 Assume that  $x^{0}_{\hour'}$ and $ y^{0}_{\hour'}$ are positive, and choose $\epsilon$ such that $\min(x^{0}_{\hour'}, y^{0}_{\hour'})>\epsilon>0 $. We construct a solution for Problem~\eqref{eq:schiever_lemma_multiplier} with the given $\epsilon$, denoted by $(x^{\epsilon,\#}_\hour,y^{\epsilon,\#}_\hour,u^{\epsilon,\#}_\hour)_{\hour \in \HOUR}$, in the following manner
\begin{equation}
\begin{cases} 
y^{\epsilon,\#}_\hour = y^{0}_\hour \eqsepv  \forall \hour \in \HOUR \setminus \{\hour'\}, \\ 
 y^{\epsilon,\#}_{\hour'} = y^{0}_{\hour'} - \frac{a_{\hour'}}{a_{\hour'} + b_{\hour'}}\epsilon  \eqfinv \\
u^{\epsilon,\#}_\hour = u^{0}_\hour \eqsepv  \forall \hour \in \HOUR, \\ 
x^{\epsilon,\#}_\hour = x^{0}_\hour \eqsepv  \forall \hour \in \HOUR\setminus\{\hour'\}, \\ 
x^{\epsilon,\#}_{\hour'} = x^{0}_{\hour'}-\frac{b_{\hour'}}{a_{\hour'} + b_{\hour'}}\epsilon \eqfinp \nonumber
\end{cases} 
\end{equation}

It is immediate to see that $(x^{\epsilon,\#}_\hour,y^{\epsilon,\#}_\hour,u^{\epsilon,\#}_\hour)_{\hour \in \HOUR}$ is feasible for Problem~\eqref{eq:schiever_lemma_multiplier} as it differs from $(x^{0}_\hour,y^{0}_\hour,u^{0}_\hour)_{\hour \in \HOUR}$ only for $\hour'$ with  $x^{\epsilon,\#}_{\hour'} \geq 0$, $y^{\epsilon,\#}_{\hour'} \geq 0, (y^{\epsilon,\#}_\hour)_{\hour \in \HOUR} \in C_y$, and $ x^{\epsilon,\#}_{\hour'}+ y^{\epsilon,\#}_{\hour'} + \epsilon = x^{0}_{\hour'}+ y^{0}_{\hour'} \geq g_{\hour'}(u^{0}_{\hour'}) = g_{\hour'}(u^{\epsilon,\#}_{\hour'})$.
Moreover, we have

\begin{equation}
\sum_{\hour \in \HOUR} a_\hour x^{\epsilon,\#}_\hour -b_\hour y^{\epsilon,\#}_\hour =  \sum_{\hour \in \HOUR} a_\hour x^{0,*}_\hour -b_\hour y^{0,*}_\hour- \frac{a_{\hour'}b_{\hour'}}{a_{\hour'} + b_{\hour'}}\epsilon + \frac{b_{\hour'} a_{\hour'}}{a_{\hour'} + b_{\hour'}}\epsilon  =  \sum_{\hour \in \HOUR} a_\hour x^{0}_\hour -b_\hour y^{0}_\hour \eqfinp \nonumber
\end{equation}
and since $(x^{\epsilon,\#}_\hour,y^{\epsilon,\#}_\hour,u^{\epsilon,\#}_\hour)_{\hour \in \HOUR}$ is feasible for Problem~\eqref{eq:schiever_lemma_multiplier}, we have 
\begin{align}\label{eq:lambda_upper_bound}
\psi_{\hour'}(\epsilon) &\leq  \sum\limits_{\hour \in \HOUR}
    \Big(\alpha_\hour x^{\epsilon,\#}_\hour +\beta_\hour y^{\epsilon,\#}_\hour\Big)+ f\big((u^{\epsilon,\#}_\hour)_{\hour \in \HOUR}\big) +\FinalCost \big(\sum_{\hour \in \HOUR} a_\hour x^{\epsilon,\#}_\hour -b_\hour y^{\epsilon,\#}_\hour \big) \nonumber \\ &= \sum\limits_{\hour \in \HOUR} \Big(\alpha_\hour x^{0}_\hour +\beta_\hour y^{0}_\hour\Big)+ f\big((u^{0}_\hour)_{\hour \in \HOUR}\big) +  \FinalCost \big(\sum_{\hour \in \HOUR} a_\hour x^{0}_\hour -b_\hour y^{0}_\hour\big) -\frac{\alpha_{\hour'} b_{\hour'}}{a_{\hour'} + b_{\hour'}}\epsilon - \frac{\beta_{\hour'}a_{\hour'}}{a_{\hour'} + b_{\hour'}}\epsilon \nonumber
    \\ 
    &= \psi_{\hour'}(0)- \big(\frac{ \alpha_{\hour'}b_{\hour'}}{a_{\hour'} + b_{\hour'}} + \frac{\beta_{\hour'}a_{\hour'}}{a_{\hour'} + b_{\hour'}} \big) \epsilon\eqfinp  
\end{align}
Moreover, as $\psi_{\hour'}(\epsilon)$ is a convex function \cite{rockaffelar}, we have
\begin{equation}
  \psi_{\hour'}(z) +\partial \psi_{\hour'}(z)(\epsilon-z) \leq \psi_{\hour'}(\epsilon) \eqsepv \forall z \geq 0 \eqfinp \nonumber
\end{equation}
In particular, for $z=0$, we have $\psi_{\hour'}(0) +\partial \psi_{\hour'}(0)\epsilon \leq \psi_{\hour'}(\epsilon)$.

Since $-\lambda^{0}_{\hour'}$ is a subgradient of $\psi_{\hour'}$ for $\epsilon=0$, it follows that
\begin{equation}\label{eq:lambda_lower_bound}
  \psi_{\hour'}(0) -\lambda^{0}_{\hour'}\epsilon \leq \psi_{\hour'}(\epsilon) \eqfinp
\end{equation}
Combining Equation~\eqref{eq:lambda_upper_bound} and Equation~\eqref{eq:lambda_lower_bound}, we get that $\lambda^{0}_{\hour'} \geq \frac{ \alpha_{\hour'} b_{\hour'}}{a_{\hour'} + b_{\hour'}} + \frac{\beta_{\hour'}a_{\hour'}}{a_{\hour'} + b_{\hour'}}$.

This ends the proof.
\end{proof}

\end{document}